\documentclass[a4paper,abstract=true,DIV=default,toc=index]{scrartcl}
\KOMAoption{overfullrule}{true}
\KOMAoption{twoside}{true}

\usepackage[normalem]{ulem}

\usepackage{amsmath}
\usepackage{amsthm}
\usepackage{amsfonts}
\usepackage{cite}
\usepackage{color,enumitem,graphicx}

\usepackage[english]{babel}
\usepackage[dvipsnames]{xcolor}
\usepackage{hyperref}
\usepackage{cleveref}

\newcommand\myshade{85}
\colorlet{mylinkcolor}{violet}
\colorlet{mycitecolor}{YellowOrange}
\colorlet{myurlcolor}{Aquamarine}

\hypersetup{
	linkcolor  = mylinkcolor!\myshade!black,
	citecolor  = mycitecolor!\myshade!black,
	urlcolor   = myurlcolor!\myshade!black,
	colorlinks = true,
}

\usepackage[T1]{fontenc}

\usepackage[useregional]{datetime2}

\usepackage{fixmath}

\numberwithin{equation}{section}

\newtheorem{theorem}{Theorem}[section]
\newtheorem{lemma}[theorem]{Lemma}

\theoremstyle{definition}
\newtheorem{definition}[theorem]{Definition}
\theoremstyle{remark}
\newtheorem{remark}[theorem]{Remark}
\newtheorem{ex}[theorem]{Example}

\newcommand{\R}{\mathbb{R}}
\newcommand{\C}{\mathbb{C}}

\usepackage{todonotes}

\usepackage{todonotes}

\newcommand{\N}{\mathbb N}

\numberwithin{equation}{section}

\DeclareMathOperator*{\supp}{supp}

\DeclareSymbolFont{rsfs}{U}{rsfs}{m}{n}
\DeclareSymbolFontAlphabet{\mathscr}{rsfs}

\newcommand{\Hs}{H^{s}(\R^N)}
\newcommand{\Hr}{H^{s}_r(\R^N)}
\newcommand{\ldue}{L^{2}(\R^N)}
\newcommand{\p}{\mathcal{P}}

\newcommand{\Pm}{\mathcal{P}_m}

\DeclareMathOperator*{\dist}{dist}

\begin{document}

\title{Normalized solutions for the fractional NLS with mass
  supercritical nonlinearity\footnote{Both authors are supported by
    GNAMPA, project ``Equazioni alle derivate parziali: problemi e
    modelli.''}}

\author{Luigi Appolloni\thanks{Dipartimento di Matematica e
    Applicazioni, Universit\`a degli Studi di Milano Bicocca. Email:
    \href{mailto:l.appolloni1@campus.unimib.it}{l.appolloni1@campus.unimib.it}}
  \qquad Simone Secchi\thanks{Dipartimento di Matematica e
    Applicazioni, Universit\`a degli Studi di Milano Bicocca. Email:
    \href{mailto:simone.secchi@unimib.it}{simone.secchi@unimib.it}}}

\date{\DTMnow}

\maketitle

\begin{abstract}
  We investigate the existence of solutions to the fractional
  nonlinear Schr\"{o}dinger equation \((-\Delta)^s u = f(u)-\mu u\) with
  prescribed \(L^2\)-norm \(\int_{\mathbb{R}^N} |u|^2 \, dx =m\) in
  the Sobolev space \(H^s(\mathbb{R}^N)\). Under fairly general
  assumptions on the nonlinearity \(f\), we prove the existence of a
  ground state solution and a multiplicity result in the radially
  symmetric case.
\end{abstract}


\section{Introduction}

In this paper we investigate the existence of solutions to the
fractional Nonlinear Schr\"{o}dinger Equation (NLS in the sequel)
\begin{gather}
  \label{eq:evol}
  \mathrm{i}\frac{\partial \psi}{\partial t}=(-\Delta)^s \psi -
  V(|\psi|)\psi,
\end{gather}
where $\mathrm{i}$ denotes the imaginary unit and
$\psi=\psi(x,t)\colon \R^N \times (0,\infty) \to \C$. This type of
Schr\"{o}dinger equation was introduced by Laskin in \cite{MR1948569},
and the interest in its analysis has grown over the years. An important
family of solutions, known under the name of \emph{travelling} or
\emph{standing waves}, is characterized by the \emph{ansatz}
\begin{gather} \label{eq:ansatz}
\psi(x,t)=e^{\mathrm{i} \mu t}u(x)
\end{gather}
for some (unknown) function \(u \colon \R^N \to \R\). These
solutions have the remarkable property that they conserve their mass
along time, i.e. \(\Vert \psi(t) \Vert_{L^2(\R^N)}\) is a constant function of \(t \in (0,+\infty)\).
It is therefore natural and meaningful to
seek solutions having a \emph{prescribed} \(L^2\)-norm. 

Coupling \eqref{eq:evol} with \eqref{eq:ansatz}, we arrive at the
problem
\begin{equation*}
  \begin{cases} (-\Delta)^s u = V(|u|)u - \mu u &\hbox{in}\ \R^N, \\
    \| u \|_{\ldue}^2=m, 
\end{cases}
\end{equation*}
where $s \in (0,1)$, $N > 2s$, $\mu \in \R$, \(m>0\) is a prescribed
parameter, and $(-\Delta)^s$ denotes the usual fractional
laplacian. We recall that
\begin{gather*}
  (-\Delta)^su(x) =C(N,s) \, \lim_{\epsilon \to
    0^+}\int_{\R^N\setminus B_\epsilon(0)}
  \frac{u(x)-u(y)}{|x-y|^{N+2s}} \, dy,
\end{gather*}
where
\begin{gather*}
  C(N,s)=\left( \int_{\R^N}\frac{1-\cos\zeta_1}{|\zeta|^{n+2s}} \, d
    \zeta \right)^{-1}.
\end{gather*}
For further details on the fractional laplacian we refer to
\cite{MR2944369}. For our purposes, and since the parameter \(s\) is
kept fixed, we will always work with a \emph{rescaled} fractional
operator, in such a way that $C(N,s)=1$.

In order to ease notation, we will write \(f(u)=V(|u|)u\), and study
the problem
\begin{equation} \label{eq:Pm}
\tag{$P_m$}
\begin{cases} (-\Delta)^s u = f(u) - \mu u &\hbox{in}\ \R^N, \\
  \| u \|_{\ldue}^2=m.
\end{cases}
\end{equation}
The r\^{o}le of the real number~\(\mu\) is twofold: it can either be
\emph{prescribed}, or it can arise as a \emph{suitable} parameter
during the analysis of \eqref{eq:Pm}.  In the present work we will
choose the second option, and \(\mu\) will arise as a Lagrange
multiplier.

Since we are looking for \emph{bound-state} solutions whose
\(L^2\)-norm must be finite, it is natural to build a 
variational setting for \eqref{eq:Pm}. Since this is by now standard,
we will be sketchy. We introduce the fractional Sobolev space
\begin{gather*}
  H^s(\R^N) = \left\lbrace u \in L^2 (\R^N) \mid [u]_{H^s(\R^N)}^2 <
    +\infty \right\rbrace,
\end{gather*}
where
\begin{gather*}
  \left[ u \right]_{\Hs}^2=\int_{\R^{N} \times
    \R^N}\frac{|u(x)-u(y)|^2}{|x-y|^{N+2s}} dx\, dy
\end{gather*}
is the so-called Gagliardo semi-norm. The norm in $H^s(\R^N)$ is
defined by
\begin{gather*}
  \Vert u \Vert = \sqrt{\strut \|u\|_{L^2}^2 + [u]^2_{H^s(\R^N)}},
\end{gather*}
which naturally arises from an inner product.
We then (formally) introduce the energy functional
\begin{gather*}
  I(u)=\frac{1}{2} \left[u\right]_{\Hs}^2-\int_{\R^N} F(u) \, dx
\end{gather*}
where $F(t)= \int_0^t f(\sigma) \, d\sigma$.  A standard approach for
studying \eqref{eq:Pm} consists in looking for critical points of~$I$
constrained on the sphere
\begin{gather*}
  S_m = \left\lbrace u \in H^s(\R^N) \mid \int_{\R^N} |u|^2 \, dx = m
  \right\rbrace.
\end{gather*}
The convenience of this variational approach depends strongly on the
behavior of the nonlinearity~\(f\). If \(f(t)\) grows slower than
\(|t|^{1+\frac{4s}{N}}\) as \(t \to +\infty\), then \(I\) is coercive
and bounded from below on \(S_m\): this is the \emph{mass subcritical
  case}, and the minimization problem
\begin{gather*}
\min \left\lbrace I(u) \mid u \in S_m \right\rbrace
\end{gather*}
is the natural approach.
On the other hand, if \(f(t)\) grows faster than
\(|t|^{1+\frac{4s}{N}}\) as \(t \to +\infty\) then \(I\) is unbounded
from below on \(S_m\), and we are in the \emph{mass supercritical
  case}. Since constrained minimizers of~\(I\) on~\(S_m\) cannot
exist, we have to find critical points at higher levels.

\bigskip

When \(s=1\), i.e. when the fractional Laplace operator
\((-\Delta)^s\) reduces to the \emph{local} differential operator
\(-\Delta\), the literature for \eqref{eq:Pm} is huge. The particular
case of a combined nonlinearity of power type, namely
\(f(t)=t^{p-2}+\mu t^{q-2}\) with \(2<q<p<2N/(N-2)\) has been widely
investigated. The interplay of the parameters \(p\) and \(q\) add some
richness to the structure of the problem.

The situation is different when~\(0<s<1\), and few results
are available. Feng \emph{et al.} in \cite{Feng2019NormalizedGS} deal
with particular nonlinearities. Stanislavova \emph{et al.} in
\cite{Stanislavova2020} add the further complication of a
trapping potential. In the recent paper~\cite{MR4095416} the author
proves some existence and asymptotic results for the fractional NLS
when a lower order perturbation to a mass supercritical pure power in
the nonlinearity is added. It is also worth mentioning
\cite{MR4135640}, where Zhang \emph{et al.} studied the problem when
the nonlinear term consists in the sum of two pure powers of different
order. They provide some existence and non-existence results analysing
separately what happens in the mass subcritical and supercritical case
for both the leading term and the lower order perturbation.

\medskip

Very recently, Jeanjean \emph{et al.} in~\cite{MR4150876}
provided a thorough treatment of the local case \(s=1\) via a careful
analysis based on the Pohozaev identity. In the present paper we partially
extend their results to the non-local case \(0<s<1\). Since we
deal with a fractional operator, our conditions on \(f\) must be
adapted correspondingly.

\medskip

We collect here our standing assumptions about the nonlinearity \(f\);
we recall that
\begin{gather*}
  F(t)=\int_0^tf(\sigma)\, d\sigma
\end{gather*}
and define the auxiliary function
\begin{gather*}
  \tilde{F}(t)=f(t)t-2F(t).
\end{gather*}
\begin{description}
\item[$(f_0)$] $f\colon \R \to \R$ is an odd and continuous function;
\item[$(f_1)$] $\lim\limits_{t \to 0} \dfrac{f(t)}{|t|^{1+4s/N}}=0$;
\item[$(f_2)$] $\lim\limits_{t \to +\infty} \dfrac{f(t)}{|t|^{(N+2s)/(N-2s)}}=0$;
\item[$(f_3)$] $\lim\limits_{t \to +\infty} \dfrac{F(t)}{|t|^{2 + 4s/N}}=+\infty$;
\item[$(f_4)$] The
  function~$t \mapsto \dfrac{\tilde{F}(t)}{|t|^{2+4s/N}}$ is strictly
  decreasing on $(-\infty,0)$ and strictly increasing on
  $(0, +\infty)$;
\item[$(f_5)$] $f(t)t< \frac{2N}{N-2s}F(t)$ for all
  $t \in \R \setminus \lbrace 0 \rbrace$;
\item[$(f_6)$] $\lim\limits_{t \to 0} \dfrac{tf(t)}{|t|^{2N/(N-2s)}}=+\infty$.
\end{description}
\begin{remark}
  The oddness of \(f\) is necessary in order to use the classical
  genus theory and to get a desired property on the fiber map that we
  will introduce in detail in the next section (see for instance Lemma
  \ref{lemma4} below). Assumption~$(f_2)$ guarantees a Sobolev
  subcritical growth, whereas $(f_3)$ characterises the problem as
  mass supercritical. At one point we will need $(f_5)$ to establish
  the strict positivity of the Lagrange multiplier $\mu$.
\end{remark}
\begin{ex}
As suggested in \cite{MR4150876}, an explicit example can be
constructed as follows. Set \(\alpha_{N,s}=\frac{4s^2}{N(N-2s)}\) for
simplicity, and define
\begin{gather*}
  f(t) = \left( \left( 2 + \frac{4s}{N} \right) \log \left(
      1+|t|^{\alpha_{N,s}} \right) + \frac{\alpha_{N,s} |t|^{\alpha_{N,s}}}{1+|t|^{\alpha_{N,s }}}
    \right) |t|^{\frac{4s}{N}} t
\end{gather*}
\end{ex}
We briefly outline our results. Firstly, we show that the ground state
level is attained with a strictly positive Lagrange multiplier.
\begin{theorem} \label{th1}
  Assume that $f$ satisfies $(f_0)$-$(f_5)$. Then \eqref{eq:Pm} admits a
  positive ground state for any~$m > 0$. Moreover, for any ground
  state the associated Lagrange multiplier $\mu$ is positive.
\end{theorem}
Furthermore, we can prove some remarkable properties of the ground
state level energy with respect the variable $m$ and its asymptotic
behaviour.
\begin{theorem} \label{th2} Assume that~$f$ satisfies
  $(f_0)$-$(f_6)$. Then the function $m \mapsto E_m$ is positive,
  continuous, strictly decreasing.  Furthermore,
  $\lim_{m \to 0^+} E_m = +\infty$ and $\lim_{m \to \infty} E_m = 0$.
\end{theorem}
Finally, we have a multiplicity result for the radially symmetric
case.
\begin{theorem} \label{th3}
  If $(f_0)$-$(f_5)$ hold and \(N >2\), then \eqref{eq:Pm}
  admits infinitely many radial solutions $(u_k)_k$ for any $m>0$. In
  particular,
  \begin{gather*}
    I(u_{k+1})\geq I(u_k)
  \end{gather*}
  for all~$k \in \N$ and $I(u_k) \to +\infty$ as $k \to +\infty$.
\end{theorem}
Our paper is organised as follows. Section \ref{section2} contains the
proofs of some preliminary lemmas that will be useful during the whole
remaining part of the paper. Moreover, we introduce a fiber map that
will play a crucial role for our purposes.  In Section~\ref{section3}
we define the ground state level energy for a fixed mass $m$ and we
start analysing its asymptotic behaviour near zero and
infinity. Section \ref{section4} is devoted to prove our main
existence theorem. Using a min-max theorem of linking type and the
fiber map cited previously, we construct a Palais-Smale sequence whose
value on the Pohozaev functional is zero and we show that a sequence
of this kind must be necessarily bounded. Finally, in Section
\ref{section5}, for the sake of completeness, we discuss the existence
of radial solutions. Here, we use a variant of the min-max theorem
already cited in Section \ref{section4}, but this time we are helped
by the fact that the space of the radially symmetric functions with
finite fractional derivative is compactly embedded in $L^p(\R^N)$ for
$p \in (2,2^*_s)$.

\section{Preliminary results} \label{section2}

We define the \emph{Pohozaev manifold}
\begin{gather*}
  \mathcal{P}_m = \left\{ u \in S_m \mid P(u)=0 \right\},
\end{gather*}
where
\begin{gather*}
  P(u) = \left[ u \right]_{H^s(\R^N)}^2 - \frac{N}{2s} \int_{\R^N}
  \tilde{F}(u) \, dx.
\end{gather*}
Let us collect some technical results that we will frequently used in
the paper. The first two Lemmas will be proved in the Appendix. We use
the shorthand
\begin{gather*}
  B_m = \left\{ u \in H^s(\mathbb{R}^N) \mid \Vert u
    \Vert^2_{L^2(\mathbb{R}^N)} \leq m \right\}.
\end{gather*}
\begin{lemma} \label{lemma1}
  Assuming $(f_0), \, (f_1), \, (f_2)$, the
  following statements hold
\begin{description}
\item[$(i)$] for every $m >0$ there exists $\delta>0$ such that
\begin{gather*}
\frac{1}{4} \left[ u \right]_{\Hs}^2 \leq I(u) \leq \left[ u \right]_{\Hs}^2 
\end{gather*}
where $u \in B_m$ and $\left[ u \right]_{\Hs} \leq \delta$.
\item[$(ii)$]Let $(u_n)_n$ be a bounded sequence in $\Hs$. If $\lim_{n \to +\infty}\Vert u_n \Vert_{L^{2+4s/N}(\R^N)}=0$ we have that
\begin{gather*}
\lim_{n\to +\infty} \int_{\R^N} F(u_n)\, dx=0=\lim_{n\to +\infty}\int_{\R^N}\tilde{F}(u_n)\, dx.
\end{gather*}
\item[$(iii)$] Let $(u_n)_n$, $(v_n)_n$ two bounded sequences in $\Hs$. 
If $\lim_{n \to +\infty} \Vert v_n \Vert_{L^{2+4s/N}}=0$ then
\begin{gather*}
\lim_{n\to +\infty} \int_{\R^N} f(u_n)v_n \, dx = 0.
\end{gather*}
\end{description}
\end{lemma}

\begin{remark}\label{oss1}
  An inspection of the proof of this Lemma shows that the inequality
  \begin{gather*}
    \int_{\R^N} \tilde{F}(u) \, dx \leq \frac{s}{N}\left[ u
    \right]_{\Hs}^2
  \end{gather*}
  holds true if $u \in B_m$ and $\left[ u \right]_{\Hs} \leq
  \delta$. It follows that
  \begin{gather*}
    P(u) \geq \frac{1}{2}\left[ u \right]_{\Hs}^2
  \end{gather*}
  for every $u \in B_m$ with $\left[ u \right]_{\Hs} \leq
  \delta$.
\end{remark}
In order to prove the next result we introduce for every $u \in \Hs$
and $\rho \in \R$ the scaling map
\begin{gather*}
  (\rho * u)(x)= e^{\frac{N\rho}{2}}u(e^{\rho}x) \quad x \in \R^N.
\end{gather*}
It easy to verify that $\rho*u \in \Hs$ and
$\Vert \rho * u\Vert_{\ldue}=\Vert u \Vert_{\ldue}$.

\begin{lemma}\label{lemma2}
  Assuming $(f_0), \, (f_1), \, (f_2)$ and $(f_3)$, we have:
\begin{description}
\item[$(i)$] $I(\rho * u) \to 0^+$ as $\rho \to - \infty$,
\item[$(ii)$] $I(\rho * u) \to - \infty$ as $ \rho \to \infty$.
\end{description}
\end{lemma}

\begin{remark}\label{oss2}
  Assume $f \in C(\R, \R)$, $(f_1)$ and $(f_4)$. Then the function
  $g\colon \R \to \R$ defined as
\begin{gather*}
g(t)=
\begin{cases}
\frac{f(t)t-2F(t)}{|t|^{2+\frac{4s}{N}}} & t \neq 0 \\
0 & t=0
\end{cases}
\end{gather*}
is continuous, strictly increasing in $(0,\infty)$ and
strictly decreasing in $(-\infty, 0)$.
\end{remark}
\begin{lemma}\label{lemma3}
Assuming $f \in C(\R, \R)$, $(f_1)$, $(f_3)$ and $(f_4)$, we have
\begin{description}
\item[$(i)$] $F(t)>0$ if $t\neq 0$,
\item[$(ii)$] there exists $(\tau_n^+)_n\subset \R^+$ and $(\tau_n^-)_n \subset \R^-$, $|\tau_n^{\pm}|\to 0$ as $n\to +\infty$ such that
\begin{gather*}
f(\tau_n^{\pm})\tau_n^{\pm} >\left(2 + \frac{4s}{N} \right)F(\tau^{\pm}_n) \quad n\geq 1,
\end{gather*}
\item[$(iii)$] there exists $(\sigma^{+}_n)_n\subset \R^+$ and $(\sigma_n^-)_n \subset \R^-$, $|\tau_n^{\pm}|\to \infty$ as $n\to +\infty$ such that
\begin{gather*}
f(\sigma_n^{\pm})\sigma_n^{\pm} >\left(2 + \frac{4s}{N} \right)F(\sigma^{\pm}_n) \quad n\geq 1,
\end{gather*}
\item[$(iv)$]
\begin{gather*}
f(t)t>\left(2+\frac{4s}{N}\right)F(t) \quad t\neq 0.
\end{gather*}
\end{description}
\end{lemma}
\begin{proof}
  $(i)$ By contradiction suppose $F(t_0)\leq 0$ for some
  $t_0\neq0$. Because of $(f_1)$ and $(f_3)$ the function
  $F(t)/|t|^{2+4s/N}$ must attain its global minimum in a point
  $\tau \neq 0$ such that $F(\tau)\leq 0$. It follows that
\begin{equation} \label{eq5} \left. \frac{d}{dt} \frac{F(t)}{|t|^{2+
        \frac{4s}{N}}}
  \right|_{t=\tau}=\frac{f(\tau)\tau-\left(2+\frac{4s}{N}\right)F(\tau)}{|\tau|^{3+\frac{4s}{N}}\operatorname{sgn}(\tau)}=0.
\end{equation}
From Remark \ref{oss2} it follows that $f(t)t>2F(t)$ if $t\neq
0$. Indeed, were the claim false, there would exists $\overline{t}$
such that $f(\overline{t})\overline{t}\leq 2F(\overline{t})$. Choosing
without loss of generality $\overline{t}<0$, we have that
$g(\overline{t})\leq 0$. This and the fact that $g(0)=0$ show that $g$
must be strictly increasing on an interval between $\overline{t}$ and
$0$. Finally, we can have a contradiction observing that
\begin{gather*}
0<f(\tau)\tau-2F(\tau)=\frac{4s}{N}F(\tau)\leq 0.
\end{gather*}
$(ii)$ We start with the positive case. By contradiction we suppose there is $T_{\alpha}>0$ small enough such that
\begin{gather*}
f(t)t\leq \left(2+\frac{4s}{N}\right)F(t)
\end{gather*}
for every $t \in \left(0,T_{\alpha}\right]$. Remembering the
expression of \eqref{eq5} computed in the step $(i)$ we have that
the derivative of $F(t)/|t|^{2+4s/N}$ is nonpositive on
$\left(0,T_{\alpha}\right]$, then
\begin{gather*}
  \frac{F(t)}{t^{2+ \frac{4s}{N}}}\geq
  \frac{F(T_{\alpha})}{T_{\alpha}^{2+ \frac{4s}{N}}} > 0\quad
  \hbox{for every} \quad t \in \left(0,T_{\alpha}\right],
\end{gather*}
that is in contradiction with $(f_1)$. The negative case is similar. 

$(iii)$ Being the two cases similar, we will prove only the negative one. Again, by contradiction we suppose there is $T_{\gamma}>0$ such that 
\begin{gather*}
f(t)t\leq \left( 2 + \frac{4s}{N}\right)F(t) \quad \hbox{for every} \quad t\leq - T_{\gamma}.
\end{gather*}
Since the derivative of $F(t)/|t|^{2+4s/N}$ is nonnegative on $\left(-\infty, -T_{\gamma}\right]$, we can deduce
\begin{gather*}
  \frac{F(t)}{|t|^{2+ \frac{4s}{N}}} \leq
  \frac{F(-T_{\gamma})}{T_{\gamma}^{2+ \frac{4s}{N}}} \quad \hbox{for
    every} \quad t \in \left(-\infty,-T_{\gamma}\right],
\end{gather*}
which contradicts~$(f_3)$.

$(iv)$ We start proving that the inequality holds weakly. By
contradiction we assume
\begin{gather*}
  f(t_0)t_0<\left(2+\frac{4s}{N}\right)F(t_0)
\end{gather*}
for some $t_0\neq 0$ and without loss of generality we can suppose
$t_0 <0$. By step $(ii)$ and $(iii)$ there are
$\tau_{\mathrm{min}}, \tau_{\mathrm{max}} \in \R$, where
$\tau_{\mathrm{min}}<t_0< \tau_{\mathrm{max}}<0$ such that
\begin{equation}\label{eq6}
  f(t)t<\left(2+\frac{4s}{N}\right)F(t) \quad \hbox{for every} \quad t \in (\tau_{\mathrm{min}},\tau_{\mathrm{max}})
\end{equation}
and
\begin{equation}\label{eq7}
  f(t)t=\left(2+\frac{4s}{N}\right)F(t) \quad \hbox{for every} \quad t \in \lbrace\tau_{\mathrm{min}},\tau_{\mathrm{max}}\rbrace.
\end{equation}
By \eqref{eq6} we have
\begin{equation} \label{eq8}
  \frac{F(\tau_{\mathrm{min}})}{|\tau_{\mathrm{min}}|^{2+\frac{4s}{N}}}<\frac{F(\tau_{\mathrm{max}})}{|\tau_{\mathrm{max}}|^{2+\frac{4s}{N}}}.
\end{equation}
Besides, by \eqref{eq7} and $(f_4)$ must be
\begin{equation} \label{eq9}
  \frac{F(\tau_{\mathrm{min}})}{|\tau_{\mathrm{min}}|^{2+\frac{4s}{N}}}=\frac{N}{4s}\frac{\tilde{F}(\tau_{\mathrm{min}})}{|\tau_{\mathrm{min}}|^{2+\frac{4s}{N}}}>\frac{N}{4s}\frac{\tilde{F}(\tau_{\mathrm{max}})}{|\tau_{\mathrm{max}}|^{2+\frac{4s}{N}}}=\frac{F(\tau_{\mathrm{max}})}{|\tau_{\mathrm{max}}|^{2+\frac{4s}{N}}},
\end{equation}
and clearly \eqref{eq8} and \eqref{eq9} are in contradiction. From
what we have just proved, we have that $F(t)/|t|^{2+4s/N}$ is
non-increasing in $(-\infty,0)$ and non decreasing in
$(0,\infty)$. Hence, by virtue of $(f_4)$ the function
$f(t)/|t|^{1+4s/N}$ must necessarily be strictly increasing in
$(-\infty,0)$ and strictly decreasing in $(0,\infty)$. Then
\begin{align*}
  \left(2+\frac{4s}{N}\right)F(t)&=\left(2+\frac{4s}{N}\right)\int_0^t
                                   \frac{f(\kappa)}{|\kappa|^{1+\frac{4s}{N}}}|\kappa|^{1+\frac{4s}{N}}\,
                                   d\kappa \\
                                 &< \left(2+\frac{4s}{N}\right)\frac{f(t)}{|t|^{1+\frac{4s}{N}}}\int_0^t |\kappa|^{1+\frac{4s}{N}}\, d\kappa=f(t)t
\end{align*}
completes the proof.
\end{proof}
\begin{lemma} \label{lemma4}
  Assume $(f_0)-(f_4)$,
  $u \in \Hs \setminus \lbrace 0 \rbrace$. Then the following hold:
\begin{description}
\item[$(i)$] There is a unique $\rho(u) \in \R$ such that $P(\rho(u)*u)=0$.
\item[$(ii)$] $I(\rho(u) * u) > I(\rho * u)$ for any $\rho \neq
  \rho(u)$. Moreover $I(\rho(u)*u)>0$.
\item[$(iii)$] The map $u \to \rho(u)$ is continuous for every $u \in \Hs$.
\item[$(iv)$] $\rho(u)=\rho(-u)$ and $\rho(u( \cdot + y))=\rho(u)$ for ever $u \in \Hs\setminus \lbrace 0\rbrace$ and $y \in \R^N$. 
\end{description}
\end{lemma}
\begin{proof}
$(i)$ Since
\begin{gather*}
I(\rho*u)=\frac{1}{2} e^{2 \rho s} \left[ u \right]_{\Hs}^2-e^{-N \rho} \int_{\R^N} F(e^{N \rho}u) \, dx
\end{gather*}
it is easy to check that $I(\rho * u)$ is $C^1$ with respect to $\rho$. Now, computing
\begin{gather*}
\frac{d}{d \rho} I(\rho * u) = \rho e^{2 \rho s}\left[ u \right]_{\Hs}^2-\frac{N}{2}e^{- N \rho}\int_{\R^N} \tilde{F}\left(e^{\frac{N \rho}{2}} u\right) \, dx.
\end{gather*}
and observing that
\begin{gather*}
P(\rho * u)=e^{2 \rho s} \left[ u \right]_{\Hs}^2-\frac{N}{2s}e^{-N\rho}\int_{\R^N} \tilde{F}\left(e^{\frac{N \rho}{2}} u\right) \, dx
\end{gather*}
we deduce
\begin{gather*}
\frac{d}{d\rho} I(\rho * u)=s P(\rho * u).
\end{gather*}
Remembering that by lemma \ref{lemma2}
\begin{gather*}
\lim_{\rho \to -\infty} I(\rho * u) =0^+ \quad \hbox{and} \quad \lim_{\rho \to \infty} I(\rho * u)=-\infty
\end{gather*}
we can conclude that $\rho \mapsto I(\rho * u)$ must reach a global maximum at some point $\rho(u)$; since
\begin{gather*}
0=\frac{d}{d \rho} I(\rho(u) * u)=sP(\rho(u) * u),
\end{gather*}
we conclude that $P(\rho (u) * u)=0$. To check the uniqueness of the point $\rho(u)$, recalling the
function $g$ defined in Remark \ref{oss2}, we observe that
$\tilde{F}(t)=g(t) |t|^{2 + \frac{4s}{N}}$ for every $t \in \R$. Thus
we obtain
\begin{align*}
  P(\rho * u )& = e^{2 \rho s} \left[ u \right]_{\Hs}^2-\frac{N}{2s}e^{2 \rho s} \int_{\R^N} g(e^{\frac{N \rho}{2}}u)|u|^{2+\frac{4s}{N}} \, dx \\
              &=e^{2 \rho s} \left[ \left[ u \right]_{\Hs}^2-\frac{N}{2s} \int_{\R^N} g(e^{\frac{N \rho}{2}}u)|u|^{2+\frac{4s}{N}} \, dx \right]=\frac{1}{s}\frac{d}{d\rho}I(\rho *u).
\end{align*}
Fixing $t \in \R \setminus \lbrace 0 \rbrace$, thanks to Remark
\ref{oss2} and $(f_4)$, we notice that the function
$\rho \mapsto g\left(e^{\frac{N \rho}{2}}t\right)$ is strictly
increasing. Thus, by virtue of the previous computations, it
follows that $\rho(u)$ must be unique.

$(ii)$ This follows at once from \((i)\).

$(iii)$ By step $(i)$ the function $ u \mapsto \rho(u)$ is well
defined. Let $u \in \Hs \setminus \lbrace 0\rbrace$ and
$(u_n)_n \subset \Hs \setminus \lbrace 0\rbrace$ a sequence such that
$u_n \to u$ in $\Hs$ as $n \to +\infty$. We set $\rho_n=\rho(u_n)$ for
any $n \geq 1$. Let us show that up to a subsequence we have
$\rho_n \to \rho(u)$ as $n \to +\infty$.

\textit{Claim}. The sequence $(\rho_n)_n$ is bounded.

We recall that the function $h_{\lambda}$ defined in \eqref{eq4}
noticing that by lemma \ref{lemma3} $(i)$ $h_0(t)\geq 0$ for every
$t \in \R$. We assume by contradiction that up to a subsequence
$\rho_n \to +\infty$. By Fatou's lemma and the fact that $u_n \to u$
a.e. in $\R^N$, we have that
\begin{gather*}
  \lim_{n \to +\infty} \int_{\R^N} h_0\left(e^{\frac{N \rho_n}{2}}u_n
  \right) |u_n|^{2+\frac{4s}{N}} \, dx = \infty.
\end{gather*}
As a consequence of that, by \eqref{eq10} with $\lambda=0$ and step
$(ii)$, we obtain
\begin{equation} \label{eq11} 0 \leq e^{-2 \rho_n
    s}I(\rho_n*u_n)=\frac{1}{2}\left[u_n\right]_{\Hs}^2-\int_{\R^N}h_0
  \left(e^{\frac{N \rho_n}{2}}u_n \right)|u_n|^{2 + \frac{4s}{N}} \,
  dx \to - \infty
\end{equation}
as $n \to +\infty$ that is evidently not possible. Then $(\rho_n)_n$
must be bounded from above. Now we assume, again by contradiction, that
$\rho_n \to -\infty$. By step $(ii)$ we observe that
\begin{gather*}
  I(\rho_n*u_n) \geq I(\rho(u) * u_n),
\end{gather*}
and since $\rho(u)*u_n \to \rho(u) *u$ in $\Hs$, it follows that
\begin{gather*}
  I(\rho(u)*u_n)=I(\rho(u)*u)+o_n(1).
\end{gather*}
We deduce that
\begin{equation} \label{eq12} \liminf_{n \to +\infty} I(\rho_n * u_n)
  \geq I(\rho(u)*u)>0.
\end{equation}
Since we have $\rho_n * u_n \subset B_m$ for $m \gg 1$,
Lemma \ref{lemma1} $(i)$ implies that there exists~$\delta >0$ such that
if $\left[\rho_n * u_n\right]_{\Hs} \leq \delta $, we have
\begin{equation} \label{eq13}
  \frac{1}{4} \left[ \rho_n *
    u_n\right]_{\Hs}^2\leq I(\rho_n*u_n) \leq \left[ \rho_n *
    u_n\right]_{\Hs}^2.
\end{equation}
Since
\begin{gather*}
  \left[ \rho_n * u_n \right]_{Hs}=e^{\rho_n s}\left[ u_n
  \right]_{\Hs},
\end{gather*}
\eqref{eq13} holds for any $n$ sufficiently large. Therefore we obtain
\begin{gather*}
  \liminf_ {n \to +\infty} I(\rho_n * u_n)=0,
\end{gather*}
in contradiction to~\eqref{eq12}. The claim is proved.

The sequence~$(\rho_n)_n$ being bounded, we can assume  that, up to a
subsequence,  $\rho_n \to \rho^*$ for some $\rho^*$ in $\R$. Hence,
$\rho_n * u_n \to \rho^* * u$ in $\Hs$ and since $P(\rho_n * u_n)=0$
we have
\begin{gather*}
P(\rho^* * u)=0.
\end{gather*}
By the uniqueness proved at step $(ii)$ we obtain $\rho^*=\rho(u)$.

$(iv)$ Since $f$ is odd by $(f_0)$, the fact that
\begin{gather*}
P(\rho(u)*(-u))=P\left(-(\rho(u)*u)\right)=P(\rho(u)*u)=0
\end{gather*}
imply $\rho(u)=\rho(-u)$. Similarly, changing the variables in the integral, we can verify that \(\rho\) is invariant under translation, and it is easy to check that
\begin{gather*}
P(\rho(u)*u(\cdot + y))=P(\rho(u)*u)=0,
\end{gather*}
thus $\rho(u(\cdot + y))=\rho(u)$.
\end{proof} 
As we are going to see, the functional $I$ constrained on
\(\mathcal{P}_m\) has some crucial properties.
\begin{lemma} \label{lemma5}
  Assuming $(f_0)-(f_4)$, the following statements are true:
\begin{description}
\item[$(i)$] $\Pm \neq \emptyset$,
\item[$(ii)$] $\inf_{u \in \Pm} \left[u\right]_{\Hs}>0$,
\item[$(iii)$] $\inf_{u \in \Pm} I(u) >0$,
\item[$(iv)$] $I$ is coercive on $\Pm$, i.e. $I(u_n)\to \infty$ if
  $(u_n)_n \subset \Pm$ and $\Vert u_n \Vert_{\Hs} \to \infty$ as
  $n \to +\infty$.
\end{description}
\end{lemma}
\begin{proof}
Statement~$(i)$ follows directly from Lemma \ref{lemma4} $(i)$.

$(ii)$ Were the assertion not true, we would be able to take a
sequence $(u_n)_n \subset \Pm$ such that
$\left[u_n \right]_{\Hs}\to 0$, and so, by Lemma \ref{lemma1} $(i)$ we
could also find $\delta >0$ and $\overline{n}$ so large that
$\left[u_n \right]_{\Hs} \leq \delta$ for every $n \geq
\overline{n}$. By Remark \ref{oss1} we would have
\begin{gather*}
0=P(u_n) \geq \frac{1}{2} \left[u_n \right]_{\Hs}^2
\end{gather*}
which is possible only for a constant~$u_n$. But this is not
admissible since $u \in S_m$. Hence the statement must hold.

$(iii)$ For every $u \in \Pm$ Lemma \ref{lemma4} $(ii)$ and $(iii)$
implies that
\begin{gather*}
  I(u)=I(0*u)\geq I(\rho * u) \quad \hbox{for every} \quad \rho \in
  \R.
\end{gather*}
Let $\delta>0$ be the number given by Lemma \ref{oss1} $(i)$ and set
\(1/\rho:=s \log \left(\delta / \left[ u \right]_{\Hs}\right)\).
Since
$\delta = \left[\rho * u \right]_{\Hs}$, using again Lemma
\ref{lemma1} $(i)$ we obtain
\begin{gather*}
  I(u) \geq I(\rho * u) \geq \frac{1}{4} \left[\rho * u
  \right]_{\Hs}^2=\frac{1}{4} \delta^2
\end{gather*}
proving the statement.

$(iv)$ By contradiction we suppose the existence of $(u_n)_n \subset \Pm$
such that $\Vert u_n \Vert_{\Hs} \to \infty$ with
$\sup_{n \geq 1} I(u_n) \leq c$ for some $ c \in (0, \infty)$. For any
$n \geq 1$ we set
\begin{gather*}
  \rho_n =\frac{1}{s}\log \left(\left[u_n \right]_{\Hs} \right) \quad
  \hbox{and} \quad v_n=(-\rho_n)*u_n.
\end{gather*}
Evidently $\rho_n \to +\infty$, $(v_n)_n \subset S_m$ and
$\left[v_n \right]_{\Hs}=1$. We denote with
\begin{gather*}
  \alpha=\limsup_{n \to +\infty} \sup_{y \in
      \R^N}\int_{B(y,1)}|v_n|^2 \, dx
\end{gather*}
and we distinguish two cases.

\textit{Non vanishing:} $\alpha >0$. Up to a subsequence we can assume
the existence of a sequence $(y_n)_n \subset \R^N$ and
$\omega \in \Hs \setminus \lbrace 0\rbrace$ such that
\begin{gather*}
  \omega_n=v_n(\cdot + y_n) \rightharpoonup \omega \,\, \hbox{in} \,\,
  \Hs \quad \hbox{and} \quad \omega_n \to \omega \, \, \hbox{a.e. in}
  \, \R^N.
\end{gather*}
Recalling the definition of the continuous function $h_{\lambda}$ with $\lambda=0$,
remembering that $\rho_n \to +\infty$ as $n \to +\infty$ and using the
Fatou's lemma we have
\begin{gather*}
  \lim_{n \to +\infty} \int_{\R^N} h_0 \left( e^{\frac{N
        \rho_n}{2}}\omega_n\right)|\omega_n|^{2+\frac{4s}{N}} \, dx=
  \infty.
\end{gather*}
By step $(iii)$ and \eqref{eq5}, after changing the variables in the
integral, we obtain
\begin{align*}
0 \leq e^{-2 \rho_n s}I(u_n) &=e^{- 2 \rho_n s}I(\rho_n * v_n)= \frac{1}{2}-\int_{\R^N}h_0\left(e^{\frac{N \rho_n}{2}}v_n \right)|v_n|^{2 + \frac{4s}{N}} \, dx\\ 
& =\frac{1}{2}-\int_{\R^N}h_0\left(e^{\frac{N \rho_n}{2}} \omega \right)|\omega_n|^{2 + \frac{4s}{N}} \, dx \to -\infty
\end{align*}
as $n \to +\infty$.

\emph{Vanishing:} $\alpha=0$. By \cite[Lemma II.4]{MR3059423}, we have
that $v_n \to 0 $ in $L^{2+\frac{4s}{N}}(\R^N)$ and by Lemma
\ref{lemma1} $(ii)$ we see that
\begin{gather*}
  \lim_{n \to +\infty} e^{N \rho} \int_{\R^N}F\left(e^{\frac{N
        \rho}{2}}v_n \right)=0 \quad \hbox{for every} \quad \rho \in
  \R.
\end{gather*}
Since $P(\rho_n*v_n)=P(u_n)=0$, by Lemma \ref{lemma4} $(ii)$ and
$(iii)$, we obtain
\begin{multline*}
  c\geq I(u_n)=I(\rho_n * v_n)\\
  \geq P(\rho*v_n)= \frac{1}{2}e^{2 \rho s}- e^{-N \rho} \int_{\R^N}
  F\left(e^{\frac{N \rho}{2}}v_n\right)\, dx = \frac{1}{2}e^{2 \rho
    s}-o_n(1).
\end{multline*}
We can conclude choosing $\rho > \log(2c)/2s$ and letting
$n \to +\infty$.
\end{proof}
We conclude with a splitting result \emph{\`a la} Brezis-Lieb. A proof
is included for the reader's convenience.
\begin{lemma} \label{lemma6} Let $f\colon \R \to \R $ continuous, odd
  and let $(u_n)_n \subset \Hs$ a bounded sequence such that
  $u_n \to u$ pointwise almost everywhere in $\R^N$. If there exists
  $C >0$ such that
\begin{gather*}
|f(t)| \leq C \left(|t|+|t|^{2^*_s-1} \right),
\end{gather*}
then 
\begin{gather*}
\lim_{n \to +\infty} \int_{\R^N}|F(u_n)-F(u_n-u)-F(u)| \, dx=0
\end{gather*}
\end{lemma}
\begin{proof}
  Let \(a\), \(b \in \mathbb{R}\) and \(\varepsilon>0\). We compute
  \begin{align*}
    \left| F(a+b)-F(a) \right| &= \left| \int_0^1 \frac{d}{d\tau}
                                 F(a+\tau b)\, d\tau \right| \\
                               &= \left| \int_0^1 F'(a+\tau b) b \, d\tau \right| \\
    &\leq C \int_0^1 \left( |a+\tau b| + |a+\tau b|^{2_s^*-1} \right)
      |b| \, d\tau \\
    &\leq C \left( |a| + |b| + 2^{2_s^*-1} \left(
    |a|^{2_s^*-1}+|b|^{2_s^*-1} \right) \right) |b| \\
    &\leq C \left( |a| + |b| + 2^{2_s^*} \left(
      |a|^{2_s^*-1}+|b|^{2_s^*-1} \right) \right) |b| \\
    &\leq C \left( |ab| + b^2 + 2^{2_s^*} \left( |a|^{2_s^*-1}|b|+|b|^{2_s^*}
    \right) \right) .
  \end{align*}
  We have used that $\tau \leq 1$ and the convexity inequality
  \begin{gather*}
    |a+b|^{2_s^*-1} \leq 2^{2_s^*-1} \left(
      |a|^{2_s^*-1}+|b|^{2_s^*-1} \right).
  \end{gather*}
  Now we use Young's inequality twice:
  \begin{align*}
    |ab| &\leq \varepsilon \frac{a^2}{2} + \frac{1}{2\varepsilon}
           |b|^2 \\
    |a|^{2_s^*-1}|b| &\leq \eta^{\frac{2_s^*}{2_s^*-1}}
                       \frac{|a|^{2_s^*}}{\frac{2_s^*}{2_s^*-1}} +
                       \frac{1}{\eta^{2_s^*}} \frac{|b|^{2_s^*}}{2_s^*}.
  \end{align*}
  Hence, choosing
  \begin{gather*}
    \eta = \varepsilon^{\frac{2_s^*-1}{2_s^*}},
  \end{gather*}
  we get
  \begin{gather*}
  |ab| + b^2 + 2^{2_s^*} \left( |a|^{2_s^*-1}|b|+|b|^{2_s^*}
    \right) \leq \varepsilon \frac{a^2}{2} + \frac{1}{2\varepsilon}
              b^2 + b^2 + 2^{2_s^*} \left(
              |a|^{2_s^*-1}|b|+|b|^{2_s^*}\right)  \\
    \leq \varepsilon C \left( a^2 + |2a|^{2_s^*} \right) + C \left[
      \left(1+\varepsilon^{-1} \right) b^2 + \left( 1+
      \varepsilon^{1-2_s^*} \right) |2b|^{2_s^*} \right] \\
    = \varepsilon \varphi(a) + \psi_\varepsilon (b).
  \end{gather*}
  Applying \cite[Theorem 2]{MR699419} with $g_n=u_n-u$ and $f=u$ we
  have the assertion.
\end{proof}

\section{Behavior of the map~$m \mapsto E_m$} \label{section3}

Under our standing assumptions~$(f_0)$--$(f_4)$, for every $m>0$
we can define the least level of energy
\begin{gather*}
  E_m=\inf_{u \in \Pm} I(u).
\end{gather*}
This section is devoted to the analysis of the quantity \(E_m\) as a
\emph{function} of \(m>0\).
\begin{lemma} \label{lemma7}
  If $(f_0)$--$(f_4)$ hold true, then
  $m \mapsto E_m$ is continuous.
\end{lemma}
\begin{proof}
  Let $m>0$ and $(m_k)_k \subset \R$ such that $m_k \to m$ in $\R$. We
  want to show that $E_{m_k} \to E_m$ as $k \to +\infty$. Firstly, we
  will prove that
  \begin{equation} \label{eq14}
    \limsup_{k\to +\infty} E_{m_k} \leq
    E_m.
  \end{equation}
  For any $u \in \Pm$ we define
  \begin{gather*}
    u_k:=\sqrt{\frac{m_k}{m}} u \in S_{m_k}, \quad k \in \N.
  \end{gather*}
  It is easy to see that $u_k \to u$ in $\Hs$, thus, by Lemma
  \ref{lemma4} $(iii)$ we get

  \noindent $\lim_{k \to +\infty} \rho(u_k)=\rho(u)=0$. Therefore
  \begin{gather*}
    \rho(u_k)*u_k \to \rho(u)*u=0 \quad \hbox{in $\Hs$}
  \end{gather*}
  as $k \to +\infty$ and as a consequence
\begin{gather*}
\limsup_{k \to +\infty} E_{m_k} \leq \limsup_{k \to +\infty}I(\rho(u_k)*u_k)=I(u).
\end{gather*}
Since this holds for any~$u$, we obtain~\eqref{eq14}. The next step
consists in proving
\begin{equation} \label{eq15}
\liminf_{k \to +\infty} E_{m_k} \geq E_m.
\end{equation} 
From the definition of $E_{m_k}$, it follows that for every
$ k \in \N$ there exists $v_k \in \mathcal{P}_{m_k}$ such that
\begin{equation} \label{eq16}
I(v_k) \leq E_{m_k}+ \frac{1}{k}.
\end{equation}
We set 
\begin{gather*}
  t_k:=\left( \frac{m}{m_k} \right)^{\frac{1}{N}} \quad \hbox{and}
  \quad \tilde{v}_k:=v_k\left( \frac{\cdot}{t_k} \right) \in S_m.
\end{gather*}
By Lemma \ref{lemma4} and \eqref{eq16} we get
\begin{align*}
  E_m & \leq I(\rho(\tilde{v}_k)*\tilde{v}_k)\leq I(\rho(v_k)*\tilde{v}_k)+\left| I(\rho(\tilde{v}_k)*\tilde{v}_k) - I(\rho(\tilde{v}_k)*v_k) \right|\\
      & \leq I(v_k) + \left| I(\rho(\tilde{v}_k)*\tilde{v}_k) -
        I(\rho(\tilde{v}_k)*v_k) \right|  \\
      &\leq E_{m_k} + \frac{1}{k} + \left| I(\rho(\tilde{v}_k)*\tilde{v}_k) - I(\rho(\tilde{v}_k)*v_k) \right| \\
      & =: E_{m_k} + \frac{1}{k} + C(k). 
\end{align*}
In order to prove~\eqref{eq15} we show that
\begin{equation*} \label{eq17}
\lim_{k \to +\infty} C(k)=0.
\end{equation*}
Indeed, as a first step we notice that
$\rho * \left(v\left( \frac{\cdot}{t}\right)
\right)=(\rho*v)\left(\frac{\cdot}{t}\right)$, and after a change of
variable we get
\begin{align*}
  C(k) &=\left| \frac{1}{2}\left(t_k^{N-2s}-1\right)\left[\rho(\tilde{v}_k)*v_k\right]_{\Hs}^2-\left(t_k^N-1\right)\int_{R^N}F(\rho(\tilde{v}_k)*v_k) \, dx \right| \\
       & \leq  \frac{1}{2}\left|t_k^{N-2s}-1\right|\left[\rho(\tilde{v}_k)*v_k\right]_{\Hs}^2+\left|t_k^N-1\right|\int_{R^N}\left|F(\rho(\tilde{v}_k)*v_k)\right| \, dx \\
       &=: \frac{1}{2}\left|t_k^{N-2s}-1\right|A(k)+\left|t_k^{N}-1\right|B(k).
\end{align*}
Since $t_k \to 1$ as $k\to +\infty$, it suffices to prove that
\begin{equation}\label{eq18}
  \limsup_{k\to +\infty} A(k) < \infty, \quad \limsup_{k \to +\infty} B(k)<\infty.
\end{equation}
We divide the proof of \eqref{eq18} in three claims.

\medskip 

\emph{Claim 1:} $(v_k)_k$ is bounded in $\Hs$.

Recalling \eqref{eq14} and \eqref{eq16} we have that
\begin{gather*}
\limsup_{k \to +\infty} I(v_k) \leq E_m.
\end{gather*}
Thus, observing that $v_k \in \mathcal{P}_{m_k}$ and $m_k \to m$ if the claim does not hold, we obtain a contradiction with lemma \ref{lemma5} $(iv)$.

\medskip

\textit{Claim 2:} $(\tilde{v}_k)_k$ is bounded in $\Hs$, and there are a sequence~$(y_k)_k \subset \R$ and $v \in \Hs \setminus \lbrace 0 \rbrace$ such that~$\tilde{v}(\cdot + y_k) \to v$ a.e. in $\R^N$  up to a subsequence .

To see the boundedness of $(\tilde{v}_k)_k$ it suffices to notice that $t_k \to 1$ and the statement follows by claim 1. Now, we set
\begin{gather*}
\alpha=\limsup_{k \to +\infty} \sup_{y \in \R^N} \int_{B(y,1)} |\tilde{v}_k|^2 \, dx.
\end{gather*} 
If $\alpha=0$, by \cite[Lemma II.4]{MR3059423} we get $\tilde{v}_k \to 0$ in $L^{2+\frac{4s}{N}}(\R^N)$. As a consequence we have that
\begin{gather*}
\int_{\R^N}|v_k|^{2+\frac{4s}{N}} \, dx=\int_{\R_N} |\tilde{v}_k(t_k\cdot)|^{2+\frac{4s}{N}} \, dx = t_k^{-N}\int_{\R^N} |\tilde{v}_k|^{2+ \frac{4s}{N}} \, dx \to 0
\end{gather*}
as $k \to +\infty$, and since $P(v_k)=0$, by Lemma \ref{lemma2} $(i)$, we deduce that
\begin{gather*}
\left[v_k\right]_{\Hs}^2=\frac{N}{2s} \int_{\R^N} \tilde{F}(v_k) \, dx \to 0.
\end{gather*}
In this case, by virtue of Remark \ref{oss1}, we see that
\begin{gather*}
0=P(v_k)\geq \frac{1}{2}\left[v_k\right]_{\Hs}^2,
\end{gather*}
which is admissible only if $v_k$ in constant. But this is in contradiction
with the fact that $v_k \in \mathcal{P}_{m_k}$. Hence $\alpha$ must be
strictly positive.

\medskip 

\textit{Claim 3:} $\limsup_{k \to +\infty} \rho(\tilde{v}_k)<\infty$.

By contradiction we assume that up to a subsequence
$\rho(\tilde{v}_k) \to \infty$ as $k \to +\infty$. By Claim 2 we can
suppose the existence of a sequence $(y_k)_k \subset \R^N$ and
$v \in \Hs \setminus \lbrace 0\rbrace $ such that
\begin{equation} \label{eq19}
\tilde{v}_k(\cdot +y_k) \to v \quad \hbox{a.e. in} \, \R^N.
\end{equation}
Instead, by Lemma $\ref{lemma4}$ we get 
\begin{equation} \label{eq20}
\rho(\tilde{v}_k(\cdot+y_k))=\rho(\tilde{v}_k) \to \infty
\end{equation}
and
\begin{equation} \label{eq21}
I(\rho(\tilde{v}_k(\cdot+y_k))*\tilde{v}_k(\cdot+y_k))\geq 0.
\end{equation}
Now, taking into account \eqref{eq19}, \eqref{eq20}, \eqref{eq21} and
arguing similarly as we have already done to prove \eqref{eq11} we
have a contradiction. The proof concludes observing that by Claims 1
and 3
\begin{equation} \label{eq22}
\limsup_{k\to +\infty} \Vert \rho(\tilde{v}_k)*v_k \Vert_{\Hs} < \infty.
\end{equation}
Hence, by virtue of $(f_0)-(f_2)$ and \eqref{eq22}, \eqref{eq18} holds true.
\end{proof}
The next result provides a weak monotonicity property for \(E_m\).
\begin{lemma} \label{lemma8}
If $(f_0)-(f_4)$ hold, then $m \mapsto E_m$ is non-increasing in $ (0, \infty)$.
\end{lemma}
\begin{proof}
It suffices to show that for every $\varepsilon >0$ and $m, \, m' >0$ with $m> m'$ we have
\begin{equation} \label{eq23}
E_m \leq E_{m'} + \frac{\varepsilon}{2}.
\end{equation}
Now, we take $\chi \in C^{\infty}_c(\R^N)$ radial such that
\begin{gather*}
\chi(x)=
\begin{cases}
1 & |x| \leq 1 \\
\left[0,1\right] & 1<|x|\leq 2 \\
0 & |x|>2
\end{cases}
\end{gather*}
and $u \in \mathcal{P}_{m'}$. For every $\delta>0$ we set $u_{\delta}(x)=u(x)\chi(\delta x)$. By
a result of Palatucci \emph{et al.}, see~\cite[Lemma 5 of Section
6.1]{MR3216834}, we know that $u_{\delta} \to u$ as $\delta \to 0^+$,
and using Lemma \ref{lemma4} $(iii)$ we obtain
\begin{gather*}
\lim_{\delta \to 0^+} \rho(u_{\delta})=\rho(u)=0.
\end{gather*} 
As a consequence of that, we obtain
\begin{equation} \label{eq24}
\rho(u_{\delta})*u_{\delta} \to \rho(u)*u \quad \hbox{in} \, \Hs
\end{equation}
as $\delta \to 0^+$. Now, fixing $\delta>0$ small enough, by virtue of
\eqref{eq24} we have
\begin{equation} \label{eq25}
I(\rho(u_{\delta})*u_{\delta}) \leq I(u) + \frac{\varepsilon}{4}.
\end{equation}
After that, we choose $v \in C^{\infty}_c(\R^N)$ with $\supp(v) \subset B\left(0,1+\frac{4}{\delta}\right)\setminus B\left(0,\frac{4}{\delta}\right)$ and we set
\begin{gather*}
\tilde{v}=\frac{m-\Vert u_{\delta} \Vert^2_{L^2(\R^N)}}{\Vert v \Vert^2_{L^2(\R^N)}}
\end{gather*}
For every $\lambda \leq 0$ we also define
$ \omega_{\lambda}=u_{\delta}+\lambda *\tilde{v}$. We observe that
choosing $\lambda$ appropriately we have
\begin{gather*}
\supp (u_{\delta}) \cap \supp (\lambda * \tilde{v})= \emptyset
\end{gather*}
thus $\omega_{\lambda} \in S_m$.

\textit{Claim:} $\rho(\omega_{\lambda})$ is upper bounded as
$\lambda \to - \infty$.

If the claim does not hold we observe that by lemma \ref{lemma4}
$(ii)$ $I(\rho(\omega_{\lambda})*\omega_{\lambda})\geq 0$ and that
$\omega_{\lambda} \to u_{\delta}$ a.e. in $\R^N$ as
$\lambda \to - \infty$. Hence, arguing as we have already done to
obtain \eqref{eq11} we reach a contradiction. Then the claim must
hold.

By virtue of the claim
\begin{gather*}
  \rho(\omega_{\lambda}) + \lambda \to - \infty \quad \hbox{as} \,
  \lambda \to - \infty,
\end{gather*}
thus 
\begin{gather*}
  \left[ (\rho(\omega_{\lambda})+\lambda) * \tilde{v}\right]_{\Hs}^2=
  e^{2s (\rho (\omega_{\lambda})+\lambda)}
  \left[\tilde{v}\right]^2_{\Hs} \to 0
\end{gather*}
implying
\begin{align*}
  \Vert (\rho(\omega_{\lambda})+\lambda) * \tilde{v} \Vert_{L^{2+\frac{4s}{N}}(\R^N)} 
  \leq C \Vert (\rho(\omega_{\lambda})+\lambda) * \tilde{v} \Vert_{L^2(\R^N)} \left[(\rho(\omega_{\lambda})+\lambda) * \tilde{v}\right]_{\Hs} \to 0.
\end{align*}
As a consequence, by Lemma \ref{lemma1} $(ii)$, for a suitable
$\lambda$
\begin{equation} \label{eq26}
I((\rho(\omega_{\lambda})+\lambda) * \tilde{v})\leq \frac{\varepsilon}{4}.
\end{equation}
Finally, by Lemma \ref{lemma4} and using \eqref{eq23}, \eqref{eq25}
and \eqref{eq26} it easy to see that
\begin{align*}
  E_m & \leq I(\rho(\omega_{\lambda})*\omega_{\lambda})=I(\rho(\omega_{\lambda})*u_{\delta})+I(\rho(\omega_{\lambda})*(\lambda*\tilde{v})) \\
      & \leq I(\rho(u_{\delta})*u_{\delta})+I((\rho(\omega_{\lambda})+\lambda)*\tilde{v}) \\
      & \leq I(u)+ \frac{\varepsilon}{4}+ \frac{\varepsilon}{4} \leq E_{m'}+\varepsilon
\end{align*}
completing the proof.
\end{proof}
The strict monotonicity of \(E_m\) holds true only locally, as we now show.
\begin{lemma} 
  Assume $(f_0)-(f_4)$ hold true. Moreover, let $u \in S_m$ and
  $\mu \in \R$ such that
  \begin{gather*}
    \left(-\Delta\right)^s+\mu u=f(u)
  \end{gather*}
  and $I(u)=E_m$. Then $E_m >E_{m'}$ for every $m'>m$ close enough if
  $\mu >0$ and for any $m'<m$ close enough if $\mu <0$.
\end{lemma}
\begin{proof}
  Let $t >0$ and $\rho \in \R$. Defining
  $u_{t,\rho}:=u(\rho*(tu))\in S_{mt^2}$ and
  \begin{gather*}
    \alpha(t,\rho):=I(u_{t,\rho})=\frac{1}{2}t^2 e^{2 \rho s}\left[
      u\right]_{\Hs}^2-e^{-N \rho}\int_{\R^N}F(te^{{\frac{N
          \rho}{2}}}u) \, dx
\end{gather*}
it is straightforward to verify that
\begin{align*}
  \frac{\partial}{\partial t} \alpha (t, \rho)&= t e^{2 \rho s}\left[
                                                u\right]_{\Hs}^2-e^{-N
                                                \rho}\int_{\R^N} f
                                                \left( te^{\frac{N
                                                \rho}{2}}u \right) e^{\frac{N \rho}{2}}u \, dx \\
                                              &=t^{-1}I'(u_{t,\rho})\left[u_{t, \rho} \right].
\end{align*}
In the case~$\mu>0$, we observe that~$u_{t,\rho}\to u$ in $\Hs$ as
$(t,\rho)\to (1,0)$. Moreover, we notice that
\begin{gather*}
I'(u)\left[u\right]=-\mu \Vert u \Vert_{L^2(\R^N)}^2=-\mu m <0
\end{gather*}
and so, choosing $\delta>0$ small enough we have
\begin{gather*}
\frac{\partial\alpha}{\partial t} (t,\rho)<0 \quad \hbox{for any} \, (t,\rho) \in(1,1+\delta) \times \left[-\delta, \delta \right].
\end{gather*}
Using the Mean Value Theorem, there exists $\xi \in (1,t)$ such that 
\begin{gather*}
\frac{\partial\alpha}{\partial t}  (\xi,\rho)=\frac{\alpha(t,\rho)-\alpha(1,\rho)}{t-1}
\end{gather*}
whenever $(t,\rho) \in (1,1+\delta) \times \left[-\delta, \delta \right]$, hence
\begin{equation} \label{eq27}
\alpha(t,\rho)=\alpha(1,\rho)+(t-1)\frac{\partial}{\partial t} \alpha (\xi,\rho)<\alpha(1,\rho).
\end{equation} 
Since by Lemma \ref{lemma4} $(iii)$ $\rho(tu) \to \rho(u)=0$ as $t \to 1^+$, setting for any $m'>m$ close enough to $m$
\begin{gather*}
t:=\sqrt{\frac{m'}{m}} \in (1,1+\delta) \quad \hbox{and} \quad \rho:=\rho(tu) \in \left[-\delta, \delta \right],
\end{gather*}
and using \eqref{eq27} together with Lemma \ref{lemma4} $(ii)$ we obtain that
\begin{gather*}
E_m \leq \alpha(t,\rho(tu))<\alpha(1,\rho(tu))=I(\rho(tu)*u)\leq I(u)=E_m.
\end{gather*}
The proof for $\mu<0$ is similar, and we omit it.
\end{proof}
As a direct consequence of the previous two lemmas we have the following result.
\begin{lemma} \label{lemma10}
Assume $(f_0)-(f_4)$ hold true. In addition let $u \in S_m$ and $\mu \in \R$ such that 
\begin{gather*}
(-\Delta)^s u+\mu u=f(u)
\end{gather*} 
with $I(u)=E_m$. Then $\mu \geq 0$ and if $\mu >0$ it is $E_m>E_{m'}$
for any $m'>m>0$.
\end{lemma}
To make a step ahead, we describe the asymptotic behaviour of \(E_m\)
as \(m \to 0^+\) and  \(m \to +\infty\).
\begin{lemma} \label{lemma11}
Assume $(f_0)-(f_4)$ hold true, then $E_m \to +\infty$ as $m \to 0^+$.
\end{lemma}
\begin{proof}
  In order to prove the Lemma, we will show that for every sequence
  $(u_n)_n \subset \Hs \setminus \lbrace 0 \rbrace$ such that
\begin{gather*}
  P(u_n)=0 \quad \hbox{and} \quad \lim_{n \to +\infty} \Vert u_n
  \Vert_{L^2(\R^N)}=0
\end{gather*}
it must be $I(u_n) \to +\infty$. We set
\begin{gather*}
  \rho_n:=\frac{1}{s} \log \left(\left[u_n \right]_{\Hs} \right) \quad
  \hbox{and} \quad v_n:=(-\rho_n)*u_n
\end{gather*}
Trivially $\left[ v_n \right]_{\Hs}=1$ and
$\Vert v_n \Vert_{L^2(\R^N)}\to 0$. Moreover, thanks to these two
facts we also have by interpolation that $v_n \to 0$ in
$L^{2+\frac{4s}{N}}(\R^N)$, thus, by Lemma \ref{lemma1} $(ii)$ we have
\begin{gather*}
  \lim_{n\to +\infty} e^{-N \rho} \int_{\R^N} F\left(e^{\frac{N
        \rho}{2}}v_n \right)\, dx=0.
\end{gather*}
Since $P(\rho_n*v_n)=P(u_n)=0$, using Lemma \ref{lemma4} $(i)$ and
$(ii)$ we obtain that
\begin{align*}
  I(u_n)&=I(\rho_n*v_n) \geq I(\rho*v_n)= \frac{1}{2} e^{2 \rho s} -e^{N \rho}\int_{\R^N} F\left(e^{\frac{N \rho}{2}}v_n \right) \, dx\\
        &=\frac{1}{2}e^{2 \rho s}+ o_n(1).
\end{align*}
Since $\rho$ is arbitrary, we get the statement as $\rho \to +\infty$.
\end{proof}
\begin{lemma} \label{lemma12}
Assume $(f_0)-(f_4)$ and $(f_6)$. Then $E_m \to 0$ as $m \to +\infty$.
\end{lemma}
\begin{proof}
  We fix $u \in L^{\infty}(\R^N) \cap S_1$ and we set
  $u_m=\sqrt{m}u \in S_m$. By Lemma \ref{lemma4} $(ii)$ we can find a
  unique $\rho(m) \in \R$ such that $\rho(m) * u_m \in \Pm$. Since by
  Lemma \ref{lemma3} $(i)$ $F$ is non negative, we get
\begin{equation} \label{eq28}
0<E_m \leq I(\rho(m)*u_m)\leq \frac{1}{2}e^{2 \rho(m)s}\left[ u \right]_{\Hs}^2. 
\end{equation}
Thus, by \eqref{eq28} it suffices to show that
\begin{gather} \label{eq29}
\lim_{m \to \infty} \sqrt{m}\, e^{\rho(m)s}=0.
\end{gather}
Using the function $g$ defined in Remark \ref{oss2}, and
recalling that $P(\rho(m)*u_m)=0$ we get
\begin{gather*}
  \left[ u
  \right]_{\Hs}^2=\frac{N}{2s}m^{\frac{2s}{N}}\int_{\R^N}g\left(
    \sqrt{m}e^{\frac{N \rho(m)}{2}}u\right)|u|^{2+\frac{4s}{N}}\, dx,
\end{gather*}
which implies
\begin{equation} \label{eq30}
\lim_{m \to \infty} \sqrt{m}\, e^{\frac{N\rho(m)}{2}}=0.
\end{equation}
Now, using $(f_6)$ for any $\varepsilon>0$ we can find $\delta>0$ such
that
\begin{gather*}
  \tilde{F}(t)\geq \frac{4s}{N}F(t)\geq
  \frac{1}{\varepsilon}|t|^{\frac{2N}{N-2s}}
\end{gather*}
if $|t| \leq \delta$. Hence, taking into account the fact that
$P(\rho(m)*u_m)=0$ and \eqref{eq30}, we get
\begin{align*}
  \left[ u \right]_{\Hs}^2 & =\frac{N}{2s} \frac{1}{m}e^{-(N+2s)\rho(m)}\int_{\R^N}\tilde{F}\left(\sqrt{m}e^{\frac{N \rho(m)}{2}}u \right) \, dx\\
                           & \geq \frac{N}{2s} \frac{1}{\varepsilon}\left(\sqrt{m}e^{\rho(m)s} \right)^{\frac{4s}{N-2s}} \int_{\R^N}\tilde{F} \left(\sqrt{m}e^{\frac{N \rho(m)}{2}}u \right) \, dx
\end{align*}
for $m$ large enough. Then \eqref{eq29} holds, and the proof is complete.
\end{proof}

\section{Ground states} \label{section4}

We introduce the  functional 
\begin{gather*}
  \Psi(u)=I(\rho(u)*u)=\frac{1}{2}e^{2 \rho(u) s}\left[ u
  \right]_{\Hs}^2-e^{-N \rho(u)} \int_{\R^N}F\left( e^{\frac{N
        \rho(u)}{2}}u \right) \, dx.
\end{gather*}
\begin{lemma} \label{lemma13}
 The functional $\Psi\colon \Hs \setminus \lbrace 0 \rbrace \to \R$ is of
  class~$C^1$, and
\begin{gather*}
  d\Psi(u)\left[\varphi \right] = dI(\rho(u)*u)\left[\rho(u) * \varphi
  \right]
\end{gather*}
for every $u \in \Hs \setminus \lbrace 0 \rbrace$ and $ \varphi \in \Hs$.
\end{lemma}
\begin{proof}
  A proof appears in \cite{MR4150876} for the case
  \(s=1\). Only minor adjustments are needed in the fractional
  case, and we omit the details.
\end{proof}
For $m>0$, we consider the constrained functional \(J\colon {S_m} \to \R\)
defined by $J=\Psi_{|S_m}$. Lemma~\ref{lemma13} yields the
following statement.
\begin{lemma} 
The functional $J\colon S_m \to \R$ is $C^1$ and
\begin{gather*}
  dJ(u)\left[\varphi\right]=d\Psi(u)\left[\varphi\right]=dI(\rho(u)*u)\left[
    \rho(u) * \varphi\right]
\end{gather*}
for any $u \in S_m$ and $\varphi \in T_u S_m$, where \(T_u S_m\) is
the tangent space at \(u\) to the manifold \(S_m\).
\end{lemma}
We recall from \cite[Definition 3.1]{MR1251958} a definition that will
be useful to construct a min-max principle.
\begin{definition}
  Let $B$ be a closed subset of a metric space $X$. We say that a
  class $\mathcal{G}$ of compact subsets of $X$ is a homotopy stable
  family with closed boundary $B$ provided
\begin{description}
\item[$(i)$] every set in $\mathcal{G}$ contains $B$,
\item[$(ii)$] for any set $A$ in $\mathcal{G}$ and any homotopy
  $\eta \in C\left(\left[0,1 \right]\times X, X \right)$ that
  satisfies $\eta(t,u)=u$ for all
  $(t,u) \in \left( \lbrace 0 \rbrace \times X\right) \cup \left(
    \left[0,1 \right] \times B \right)$, one has
  $\eta \left( \lbrace 1 \rbrace \times A \right) \in \mathcal{G}$.
\end{description}
\end{definition}
We remark that $B=\emptyset$ is admissible.
\begin{lemma} \label{lemma15} Let $\mathcal{G}$ be a homotopy stable
  family of compact subset with (with $B=\emptyset$). We set
\begin{gather*}
E_{m,\mathcal{G}}=\inf_{A \in \mathcal{G}} \max_{u \in A} J(u).
\end{gather*}
If $E_{m,\mathcal{G}}>0$, then there exists a Palais-Smale sequence
$(u_n)_n \in \Pm$ for the constrained functional $I_{|S_m}$ at level
$E_{m,\mathcal{G}}$. In particular, if $\mathcal{G}$ is the class of
all singletons in $S_m$, one has that
$\Vert u_n^- \Vert_{L^2(\R^N)}\to 0$ as $n \to +\infty$.
\end{lemma}
\begin{proof}
Let $(A_n)_n \subset \mathcal{G}$ be a minimizing sequence of $E_{m,\mathcal{G}}$. We define the map
\begin{gather*}
\eta \colon \left[ 0,1 \right] \times S_m \to S_m
\end{gather*}
where $\eta(t,u)=(t\rho(u))*u$ is continuous and well defined by lemma \ref{lemma4} $(ii)$ and $(iii)$. Noticing $\eta(t,u)=u$ for every $(t,u) \in \lbrace 0 \rbrace \times S_m$ we obtain that
\begin{gather*}
D_n:=\eta(1,A_n)=\lbrace \rho(u)*u \mid u \in A_n\rbrace \in \mathcal{G}.
\end{gather*}
In particular we can see that $D_n \subset \Pm$ for any $m>0$, with $m>0$. Since $J(\rho(u)*u)=J(u)$ for every $\rho \in \R$ and $u \in S_m$, we can observe that
\begin{gather*}
\max_{u\in D_n} J(u)=\max_{u \in A_n}J(u) \to E_{m,\mathcal{G}}
\end{gather*}
thus, $(D_n)_n$ is another minimizing sequence for $E_{m,\mathcal{G}}$. Now, using \cite[Theorem 3.2]{MR1251958} we get a Palais-Smale sequence $(v_n)_n \subset S_m$ for $J$ at level $E_{m,\mathcal{G}}$ such that $\dist_{\Hs}(v_n,D_n) \to 0$ as $n \to +\infty$. We will denote
\begin{gather*}
\rho_n:=\rho(v_n) \quad \hbox{and} \quad u_n:=\rho_n*v_n.
\end{gather*}

\textit{Claim:} There exists $C>0$ such that $e^{-2\rho_n s} \leq C$ for any $n \in \N$.

We start pointing out that
\begin{gather*}
e^{-2\rho_n s}=\frac{\left[v_n\right]_{\Hs}^2}{\left[u_n\right]_{\Hs}^2}.
\end{gather*}
By virtue of the fact that $(u_n)_n \subset \Pm$, using lemma \ref{lemma5} $(ii)$ we obtain that $\left\lbrace\left[ u_n\right]_{\Hs} \right\rbrace_n$ is bounded from below. Moreover, since $D_n \subset \Pm$ and the fact that
\begin{gather*}
\max_{u \in D_n} I=\max_{u \in D_n} J \to E_{m,\mathcal{G}},
\end{gather*}
 Lemma \ref{lemma5} $(iv)$ implies that $D_n$ is uniformly bounded in $\Hs$. Finally, from $\dist(v_n,D_n) \to 0$ we can deduce that $\sup_{n\in\N} \left[ v_n \right]_{\Hs}< \infty$. Thus the claim holds.

Now, from $(u_n) \subset \Pm$  we get
\begin{gather*}
I(u_n)=J(u_n)=J(v_n) \to E_{m,\mathcal{G}}.
\end{gather*} 
Instead, for any $\psi \in T_{u_n}S_m$ we have
\begin{align*}
\int_{\R^N}v_n\left[ (-\rho_n)*\psi\right] \, dx &=\int_{\R^N}v_n e^{-\frac{N \rho_n}{2}}\psi\left(e^{-\rho_n}x \right) \, dx = \int_{\R^N} e^{\frac{N \rho_n}{2}}v_n\left(e^{\rho_n}x \right)\psi \, dx \\
&=\int_{\R^N} (\rho_n * v_n) \psi \, dx = \int_{\R^N} u_n \psi \, dx=0
\end{align*}
implying $(-\rho_n*\psi)\in T_{v_n}S_m$. Besides, by the claim
\begin{gather*}
\Vert(-\rho_n)*v_n \Vert_{\Hs} \leq \max \lbrace C,1\rbrace \Vert \psi \Vert_{\Hs}.
\end{gather*}
Denoting with $\Vert \cdot \Vert_{u,*}$ the dual norm of the space $(T_u S_m)^*$ and using Lemma \ref{eq13} we get
\begin{align*}
\Vert dI(u_n) \Vert_{u_n,*}&= \sup_{\substack{\psi \in T_{u_n} S_m \\ \left\Vert \psi \right\Vert_{\Hs} \leq 1}} |dI(u_n)\left[ \psi \right]| =\sup_{\substack{\psi \in T_{u_n}S_m \\ \Vert \psi \Vert_{\Hs}\leq 1}}  |dI(\rho_n*v_n)\left[\rho_n*((-\rho_n)*\psi) \right]| \\
&= \sup_{\substack{\psi \in T_{u_n}S_m \\ \Vert \psi \Vert_{\Hs} \leq
  1}} |dJ(v_n)\left[ (-\rho_n)*\psi \right]| \\
  &\leq \Vert dJ(v_n)\Vert_{v_n,*} \sup_{\substack{\psi \in T_{u_n}S_m \\ \Vert \psi \Vert_{\Hs} \leq 1}} \Vert(-\rho_n)*\psi \Vert_{\Hs}\\
& \leq  \max \{C, 1 \} \ \Vert dJ(v_n) \Vert_{v_n,*} \to 0
\end{align*} 
as $n \to +\infty$ remembering that $(v_n)_n$ is a Palais-Smale
sequence for the functional $J$. We have just proved $(u_n)_n$ is a
Palais-Smale sequence for the functional $I_{|S_m}$ at level
$E_{m, \mathcal{G}}$ with the additional property that
$(u_n)_n \subset \Pm$. Finally, noticing that the family of singleton
of $S_m$ is a particular homotopy stable family of compact subsets of
$S_m$, and doing this particular choice as $\mathcal{G}$, arguing
similarly as we have just done, we can obtain a minimizing sequence
$(D_n)_n$ with the additional property that its elements are non
negative: we only need to replace the functions with their absolute value. 
Moreover, $(A_n)_n$ will inherit this property, and recalling that $\dist(v_n, D_n) \to 0$ as
$n \to +\infty$ we have
\begin{gather*}
\Vert u_n^- \Vert_{L^2(\R^N)}=\Vert \rho_n*v_n^-\Vert_{L^2(\R^N)}=\Vert v_n^- \Vert_{L^2(\R^N)} \to 0.
\end{gather*} 
This concludes the proof of the lemma.
\end{proof}
\begin{lemma} \label{lemma16}
We assume $(f_0)-(f_4)$ hold. Then there exists a Palais-Smale sequence $(u_n)_n \subset \Pm$ for the constrained functional $I_{|S_m}$ at level $E_m$ such that $\Vert u_n^- \Vert_{L^2(\R^N)}\to 0$ as $n \to +\infty$.
\end{lemma}
\begin{proof}
We apply lemma $\ref{lemma15}$ with $\mathcal{G}$ the class of all singletons in $S_m$. Lemma $\ref{lemma5}$ imply that $E_m >0$, thus the only thing that remains to prove is $E_m=E_{m,\mathcal{G}}$. In order to do that, as a first step we notice that
\begin{gather*}
E_{m,\mathcal{G}}=\inf_{A \in \mathcal{G}} \max_{u \in A} J(u)=\inf_{u \in S_m} I(\rho(u)*u).
\end{gather*}
Since for every $u \in S_m$ we have that $ \rho(u) * u \in \Pm$ it must be $I(\rho(u) *u) \geq E_m$, thus $E_{m, \mathcal{G}}\geq E_m$. On the other hand, if $u \in \Pm$ we have $\rho(u)=0$ and $I(u)\geq E_{m,\mathcal{G}}$, that implies $E_m\geq E_{m,\mathcal{G}}$.
\end{proof}
\begin{lemma} \label{lemma17}
Let $(u_n)_n \subset S_m$ be a bounded Palais-Smale sequence for the constrained functional $I_{|S_m}$ at level $E_m>0$ such that $P(u_n) \to 0$ as $n \to +\infty$. Then we have the existence of $u \in S_m$ and $\mu >0$ such that, up to a subsequence and translations in $\R^N$, $u_n \to u$ strongly in $\Hs$ and
\begin{gather*}
(-\Delta)^s u+\mu u=f(u).
\end{gather*}
\end{lemma} 
\begin{proof}
  It is clear that~$(u_n)_n \subset S_m$ is bounded in $\Hs$ and is a
  Palais-Smale sequence. Together, these two facts enable us to assume
  without loss of generality that
  $\lim_{n \to +\infty} \left[ u_n \right]_{\Hs}$,
  $\lim_{n \to +\infty} \int_{\R^N} F(u_n) \, dx$, and
  $\lim_{n \to +\infty} \int_{\R^N} f(u_n)u_n \, dx$ exist. Besides,
  \cite[Lemma 3]{MR695536} implies
\begin{gather*}
(-\Delta)^su_n + \mu_n u_n-f(u_n) \to 0 \quad \hbox{in \(\Hs^*\)} 
\end{gather*} 
where we denoted
\begin{gather*}
  \mu_n=\frac{1}{m}\left(\int_{\R^N} f(u_n) u_n \, dx -\left[ u_n
    \right]_{\Hs}^2 \right).
\end{gather*}
By the assumptions done above we can see that $\mu_n \to \mu$ for some
$\mu \in \R$ and we also have that for any $(y_n)_n \subset \R^N$
\begin{equation} \label{eq33}
(-\Delta)^s u_n(\cdot + y_n)+\mu u_n(\cdot + y_n) -f(u_n(\cdot+y_n))
\to 0 \quad \hbox{in \(\Hs^*\)} .
\end{equation}

\textit{Claim:} $(u_n)_n$ is non vanishing.

Otherwise by \cite[Lemma II.4]{MR3059423} we would get $u_n \to 0$ in $L^{2+\frac{4s}{N}}(\R^N)$.
Taking into account that $P(u_n) \to 0$ and using lemma \ref{lemma1} $(ii)$ we get
\begin{gather*}
\left[ u_n \right]_{\Hs}^2=P(u_n) + \frac{N}{2s} \int_{\R^n} \tilde{F}(u_n) \, dx \to 0
\end{gather*}
and as a consequence of that,
\begin{gather*}
  E_m=\lim_{n \to +\infty} I(u_n)=\frac{1}{2} \lim_{n \to
    +\infty}\left[ u_n \right]_{\Hs}^2-\lim_{n \to +\infty}\int_{\R^N}
  F(u_n) \, dx
\end{gather*}
contradicting $E_m>0$. Then the claim must hold.

Since $(u_n)_n$ in non vanishing we can find $(y_n^1)_n \subset \R^N$
and $\omega_1 \in B_m \setminus \lbrace 0 \rbrace$ such that
$u_n(\cdot + y_n^1) \rightharpoonup \omega_1$ in $\Hs$,
$u_n(\cdot + y_n^1) \to \omega_1$ in
$L^p_{\mathrm{\mathrm{loc}}}(\R^N)$ for $p \in \left[1, 2^*_s\right]$
and $ u_n(\cdot + y_n^1) \to \omega $ a.e. in $\R^N$. Now, we want to
apply \cite[Lemma A.1]{MR695535} with $P(t)=f(t)$ and
$Q(t)=|t|^{(N+2s)/(N-2s)}$ and we notice that
\begin{multline} \label{eq34}
  \lim_{n \to +\infty} \int_{\R^N} \left| \left[ f(u_n(\cdot + y_n^1)-f(\omega_1) \right]\varphi \right| \, dx\\
  \leq \Vert \varphi \Vert_{L^{\infty}(\R^N)} \lim_{n \to +\infty}
  \int_{\supp (\varphi )} \left| f(u_n(\cdot + y_n^1)-f(\omega_1)
  \right|\, dx
\end{multline}
for any $\varphi \in C_c^{\infty}(\R^N)$. Hence, by \eqref{eq33} and
\eqref{eq34} we get
\begin{equation}
  \label{eq35} (-\Delta)^s\omega_1 + \mu \omega_1=
  f(\omega_1)
\end{equation}
and through the Pohozaev Identity (see for instance \cite[Proposition
4.1]{MR3007900}) associated to \eqref{eq35} we also have
$P(\omega_1)=0$. Now, we set $v_n^1:=u_n-\omega_1(\cdot-y_n^1)$ for
every $n \in \N$. Clearly
$v_n^1(\cdot+y_n^1)=u_n(\cdot+y_n^1)-\omega_1 \rightharpoonup 0$ in
$\Hs$, thus
\begin{align} \label{eq36}
  m&=\lim_{n \to +\infty} \Vert u_n(\cdot +
  y_n^1 ) \Vert_{L^2(\R^N)} =\lim_{n \to +\infty} \Vert
  v_n^1\Vert_{L^2(\R^N)}^2 + \Vert \omega_1 \Vert_{L^2(\R^N)}^2.
\end{align}
By lemma \ref{lemma6} we also have
\begin{gather*}
  \lim_{n \to +\infty} \int_{\R^N} F(u_n(\cdot + y_n^1)) \, dx
  =\int_{\R^N} F(\omega_1) \, dx + \lim_{n \to +\infty} \int_{\R^N}
  F(v_n^1(\cdot + y_n^1)) \, dx
\end{gather*}
hence
\begin{align} \label{eq37}
  E_m=\lim_{n \to +\infty} I(u_n)&= \lim_{n \to +\infty} I(u_n(\cdot + y_n^1)) = \lim_{n \to +\infty} I(v_n^1(\cdot + y_n^1))+I(\omega_1) \\ \notag
                                 & = \lim_{n \to +\infty} I(v_n^1) + I(\omega_1).
\end{align}

\textit{Claim:} $\lim_{n \to +\infty} I(v_n^1) \geq 0$.

If the claim does not hold, i.e $ \lim_{n \to +\infty} I(v_n^1)<0$, $(v_n^1)_n$ is non vanishing, then there exists $(y_n^2)_n \subset \R^N$ such that
\begin{gather*}
\lim_{n \to +\infty} \int_{B(y_n^2,1)} |v_n^1|^2 >0.
\end{gather*}
Since $v_n^1(\cdot + y_n^1) \to 0$ in $L^2_{\mathrm{loc}}(\R^N)$, it must be $|y_n^2-y_n^1| \to \infty$, and up to a subsequence  $v_n^1(\cdot + y_n^2) \to \omega_2$ in $\Hs$ for some $\omega_2 \in B_m \setminus \lbrace 0 \rbrace$. We notice 
\begin{gather*}
u_n(\cdot + y_n^2)=v_n^1(\cdot + y_n^2)+\omega_1(\cdot -y_n^1+y_n^2)\rightharpoonup \omega_2
\end{gather*}
thus, arguing as before, we get $P(\omega_2)=0$ and $I(\omega_2)>0$. We set
\begin{gather*} 
v_n^2= v_n^1-\omega^2 (\cdot - y_n^2)= u_n- \sum_{\ell=1}^{2}\omega_{\ell} (\cdot-y_n^{\ell})
\end{gather*}
and we observe that
\begin{align*}
\lim_{n \to +\infty}  \left[ v_n^2\right]_{\Hs}^2 &= \lim_{n \to +\infty} \left[ v_n^1 \right]_{\Hs}^2 +  \left[ \omega_2 \right]_{\Hs}^2-2 \lim_{n \to +\infty} \langle v_n^1, \omega_2(\cdot - y_n^2) \rangle_{\Hs}\\
&= \lim_{n \to +\infty} \left[ v_n^1 \right]_{\Hs}^2 +  \left[ \omega_2 \right]_{\Hs}^2 -2 \lim_{n \to +\infty} \langle v_n^1(\cdot + y_n^2), \omega_2 \rangle_{\Hs}  \\
&= \lim_{n \to +\infty}  \left[ u_n\right]_{\Hs}^2 + \left[
                                                                                                                                                                           \omega_1\right]_{\Hs}^2 - \left[ \omega_2\right]_{\Hs}^2 \\
  &\quad {}-2 \lim_{n \to +\infty} \langle u_n(\cdot + y_n^1), \omega_1\rangle_{\Hs} \\
&= \lim_{n \to +\infty}  \left[ u_n\right]_{\Hs}^2 - \sum_{\ell=1}^{2} \left[ \omega_{\ell} \right]_{\Hs}^2
\end{align*}
and 
\begin{gather*}
  0> \lim_{n \to +\infty} I(v_n^1)=I(\omega_2) + \lim_{n \to +\infty}
  I(v_n^2)> \lim_{n \to +\infty} I(v_n^2).
\end{gather*}  
Iterating, we can build an infinite sequence $(\omega_k) \subset B_m \setminus \lbrace 0 \rbrace$ such that $P(\omega_k)=0$ and
\begin{gather*}
\sum_{\ell=1}^{k} \left[ \omega_k \right]_{\Hs}^2 \leq \left[ u_n \right]_{\Hs}^2 < \infty
\end{gather*}
for every $k \in \N$. Though, this is a contradiction. Indeed,
recalling remark \ref{oss1}, for any
$\omega \in B_m \setminus \lbrace 0 \rbrace$ such that $P(\omega)=0$,
we can find $\delta >0$ such that
$\left[ \omega \right]_{\Hs}^2 \geq \delta$. Hence, the claim must
hold and $\lim_{n \to +\infty} I(v_n^1) \geq 0$.

Now, we denote with
$h:=\Vert \omega_1 \Vert_{L^2(\R^N)}^2 \in \left( 0, m\right]$. By
virtue of the claim, \eqref{eq37} and the fact that
$\omega_1 \in \mathcal{P}_h$, we get
\begin{gather*}
E_m =I(\omega_1) +\lim_{n \to +\infty} I(v_n^1) \geq I(\omega^1) \geq E_h
\end{gather*}
but, recalling that $E_m$ in non-increasing by lemma \ref{lemma8}, we obtain
\begin{equation} \label{eq38}
I(\omega_1) =E_m=E_h
\end{equation} 
and
\begin{equation} \label{eq39}
\lim_{n \to +\infty} I(v_n^1)=0.
\end{equation}
To prove that $\mu \geq 0$ it suffices to put together \eqref{eq35},
\eqref{eq38} and Lemma \ref{lemma10}. Instead, to see that
$\mu$ is strictly positive, using
$(f_5)$, lemma \ref{lemma2} and the Pohozaev Identity corresponding to
\eqref{eq35}, we get
\begin{equation} \label{eq40}
  \mu=\frac{1}{ms}\int_{\R^N} \left(
    N F(\omega_1) -\frac{N-2s}{2}f(\omega_1) \omega_1
  \right)\, dx >0.
\end{equation}
At this point, we suppose by contradiction that $h<m$, but taking
into account \eqref{eq35}, \eqref{eq40} and Lemma \eqref{lemma10} we
would have
\begin{gather*}
I(\omega_1)=E_h>E_m
\end{gather*} 
which is not compatible with \eqref{eq39}. Thus $h=m$. Moreover, by
\eqref{eq36} \(v_n^1 \to 0\) in \(L^2(\R^N)\).  It remains only to
prove the strong convergence of $(v_n^1)_n$ in $\Hs$. To do that, it
is sufficient to notice that by lemma \ref{lemma1} $(ii)$ we
have~$\lim_{n \to +\infty} \int_{\R^N} F(v_n^1) \, dx$, and so we
obtain the assertion thanks to \eqref{eq39}.
\end{proof}

\begin{proof}[Proof of theorem \ref{th1}]
  Applying lemma \ref{lemma16} we obtain a Palais-Smale sequence
  $(u_n)_n \subset \Pm $ at level $E_m>0$ for the constrained
  functional $I_{|S_m}$. This sequence is bounded in $\Hs$ by Lemma
  \ref{lemma5} and through Lemma \ref{lemma17} we get a critical point
  $u \in S_m$ at the level $E_m>0$ that results to be a ground state
  energy. Finally, since $\Vert u_n^- \Vert_{L^2(\R^N)} \to 0$ we
  deduce that $u \geq 0$ and after applying the strong maximum
  principle we obtain $u >0$.
\end{proof}
\begin{proof}[Proof of theorem \ref{th2}]
  The proof is a direct consequence of Theorem \ref{th1} and Lemmas
  \ref{lemma5}, \ref{lemma7}, \ref{lemma8}, \ref{lemma11},
  \ref{lemma12}.
\end{proof}

\section{Existence of radial solutions} \label{section5}

This section is devoted to prove the existence of infinitely many
radial solutions to problem \eqref{eq:Pm}. Before doing this, we recall
some basic definitions and we provide some notation.

\medskip

Denote by $\sigma \colon \Hs \to \Hs$ the transformation
$\sigma(u)=-u$ and let $X \subset \Hs$. A set $A \subset X$ is called
$\sigma$-invariant if $\sigma(A)=A$. A homotopy
$\eta \colon \left[0,1\right]\times X \to X$ is $\sigma$-equivariant
if $\eta(t,\sigma(u))=\sigma(\eta(t,u))$ for all
$(t,u) \in \left[0,1\right] \times X$. Next definition is in
\cite[Definition 7.1]{MR1251958}.
\begin{definition}
  Let $B$ be a closed $\sigma$-invariant subset of $X\subset \Hs$. We say
  that a class $\mathcal{G}$ of compact subsets of $X$ is a
  $\sigma$-homotopy stable family with closed boundary $B$ provided
  \begin{description}
  \item[$(i)$] every set in $\mathcal{G}$ is $\sigma$-invariant.
  \item[$(ii)$] every set in $\mathcal{G}$ contains $B$,
  \item[$(iii)$] for any set $A$ in $\mathcal{G}$ and any
    $\sigma$-equivariant homotopy
    $\eta \in C\left(\left[0,1 \right]\times X, X \right)$ that
    satisfies $\eta(t,u)=u$ for all
    $(t,u) \in \left( \lbrace 0 \rbrace \times X\right) \cup \left(
      \left[0,1 \right] \times B \right)$, one has
    $\eta \left( \lbrace 1 \rbrace \times A \right) \in \mathcal{G}$.
  \end{description}
\end{definition}
We denote with $\Hr$ the space of radially symmetric functions in
$\Hs$ and recall that $\Hr \hookrightarrow L^p(\R)$ compactly for all
$p \in (2,2^*_s)$ (see~\cite[Proposition I.1]{MR683027}).  \bigskip

In order to prove the main result of this section, we need to build a
sequence of $\sigma$-homotopy stable families of compact subsets of
$S_m\cap \Hr$. We point out that in the definition above, the case in
which $B=\emptyset$ is not excluded. The idea is borrowed from
\cite{MR4150876}. Let $(V_k)_k$ be a sequence of finite
dimensional linear subspaces of $\Hr$ such that $V_k \subset V_{k+1}$,
$\dim V_k=k$ and $\bigcup_{k\geq 1} V_k$ is dense in $\Hr$. Denote by
$\pi_k$ the orthogonal projection from $\Hr$ onto $V_k$. We recall to
the reader the definition of the genus of $\sigma$-invariant sets
introduced by M. A. Krasnoselskii and we refer to \cite[Section
7]{MR845785} or \cite[chapter 10]{MR2292344} for its basic properties.
\begin{definition}
  Let $A$ be a nonempty compact $\sigma$-invariant subset of $\Hr$. The
  genus $\gamma(A)$ of $A$ is the least integer $k$ such that there
  exists $\phi \in C(\Hr,\R^k)$ such that $\phi$ is odd and
  $\phi(x) \neq 0$ for all $x \in A$. We set $\gamma(A)=\infty$ if
  there are no integers with the above property and
  $\gamma(\emptyset)=0$.
\end{definition}
Let $\mathcal{A}$ be the family of closed $\sigma$-invariant subset of
$S_m\cap \Hr$. For each $k \in \N$, set
\begin{gather*}
  \mathcal{G}_k:=\lbrace A \in \mathcal{A} \mid \gamma(A)\geq k \rbrace
\end{gather*}
and
\begin{gather*}
  E_{m,k}=\inf_{A \in \mathcal{A}} \max_{u \in A} J(u).
\end{gather*}
Next, we give a result about the weak convergence of the nonlinearity
$f$.
\begin{lemma} \label{lemma23}
  Assume $(f_0)-(f_2)$ hold true. Let
  $(u_n)_n \subset H^s_r(\R^N)$. If $u_n \rightharpoonup u$ in
  $H^s_r(\R^N)$ for some $u \in H^s_r(\R^N)$, then
  $f(u_n) \rightharpoonup f(u)$ in $L^{\frac{2N}{N+2s}}(\R^N)$.
\end{lemma}
\begin{proof}
  We borrow some ideas from~\cite[Theorem 2.6]{MR2120260}.  We start
  exploiting the compact embeding
  $H^s_r(\R^N) \hookrightarrow L^{p}(\R^N)$ for any~$p \in
  (2,2^*_s)$. Hence, up to a subsequence, $u_n \to u$ in $L^p(\R^N)$
  and a.e. in $\R^N$. From equation \eqref{eq3}, we get
\begin{gather*}
  |f(u_n)|^{\frac{2N}{N+2s}} \leq C_{\varepsilon}
  |u_n|^{\frac{2N}{N-2s}}+C|u_n|^{2\frac{N+4s}{N+2s}}
\end{gather*}
for some $C_{\varepsilon}, \, C >0$. As a consequence of that,
recalling the fractional Sobolev inequality and observing that
$2\frac{N+4s}{N+2s} \in (2,2^*_s)$, we obtain that
$\left( f(u_n)\right)_n$ is bounded in
$L^{\frac{2N}{N+2s}}(\R^N)$. Thus, there exists
$y \in L^{\frac{2N}{N+2s}}(\R^N)$ such that
$f(u_n) \rightharpoonup y$. At this point, we fix a cover
$\left( \Omega_j \right)_j$ of $\R^N$ made of subsets with finite
measure. For any $\upsilon > 0 $, Severini-Egorov's Theorem yields the
existence of $B_{\upsilon}^j \subset \Omega_j$, with
measure~$\left| B_{\upsilon}^j \right| < \upsilon$, such that $u_n \to u$
uniformly in $\Omega_j \setminus B_{\upsilon}^j$. Clearly $y=f(u)$ in
$\Omega_j \setminus B_{\upsilon}^j$. Now, we set
\begin{gather*}
  \mathcal{Q}:=\left\lbrace x \in \R^N \mid y \neq f(u) \right\rbrace \quad
  \hbox{and} \quad Q_j:= \left\lbrace x \in \Omega_j \mid y \neq
  f(u) \right\rbrace.
\end{gather*}
Since $\upsilon $ is arbitrary and $Q_j \subset B_{\upsilon}^j$, we
have that \(Q_j\) is a set of measure zero. Furthermore, it is easy to see that
$\mathcal{Q}=\bigcup_{j=1}^{\infty}Q_j$, thus \(Q\) has measure zero and the proof is complete.
\end{proof}
From now on, we will always assume $(f_0)-(f_5)$ hold until the end of
the section. 
\begin{lemma} \label{lemma18} Let $\mathcal{G}$ be a $\sigma$-homotopy
  stable family of compact subset of $S_m\cap \Hr$ (with
  $B=\emptyset$) and set
\begin{gather*}
  E_{m,\mathcal{G}}:= \inf_{A \in \mathcal{G}} \max_{u \in A} J(u).
\end{gather*} 
If $E_{m,\mathcal{G}}>0$ then there exists a Palais-Smale sequence
$(u_n)_n$ in $\p_m \cap \Hr$ for $I_{| S_m \cap \Hr}$ at level
$E_{m,\mathcal{G}}$.
\end{lemma}
\begin{proof}
  It suffices to replace Theorem 3.2 with 7.2 of \cite{MR1251958} in
  the proof of Lemma \ref{lemma15}.
\end{proof}
\begin{lemma} \label{lemma19}
  For any $k \in \N$ we have,
\begin{description}
\item[$(i)$] $\mathcal{G}_k\neq \emptyset$ and $\mathcal{G}_k$ is a
  $\sigma$-homotopy stable family of compact subsets of $S_m\cap \Hr$
  (with $B=\emptyset$),
\item[$(ii)$] $E_{m,k+1}\geq E_{m,k}>0$.
\end{description}
\end{lemma}
\begin{proof}
  $(i)$ It suffices to notice that for any $k\in \N$ one has
  $S_m\cap V_k \in \mathcal{A}$ and that by
  \cite[Theorem 10.5]{MR2292344}
\begin{gather*}
\gamma(S_m\cap V_k)=k.
\end{gather*}
Thus $\mathcal{G}_k \neq \emptyset$. The conclusion is a direct
consequence of the definition of $\mathcal{A}$.

$(ii)$ By the previous step $E_{m,k}$ is well defined. Furthermore,
recalling that $\rho(u) * u \in \p_m$ for all $u \in A$, where $A$ is
chosen arbitrarily in $\mathcal{G}$, we have
\begin{gather*}
  \max_{u \in A} J(u)=\max I(\rho(u)*u)=\inf_{v \in \p_m} I(v),
\end{gather*}
hence $E_{m,k}>0$. The other part of the statement follows easily from
$\mathcal{G}_{k+1} \subset \mathcal{G}_{k}$.
\end{proof}
\begin{lemma} \label{lemma20}
  Let $(u_n)_n \subset S_m \cap \Hr$ be a
  bounded Palais-smale sequence for $I_{|S_m}$ at an arbitrary level
  $c>0$ satisfying $P(u_n) \to 0$. Then there exists
  $u \in S_m \cap \Hr$ and $\mu>0$ such that, up to a subsequence,
  $u_n \to u$ strongly in $\Hr$ and
\begin{gather*}
  (-\Delta)^s+\mu u=f(u).
\end{gather*}
\end{lemma}
\begin{proof}
  By the boundedness of the Palais-Smale sequence we may assume
  $u_n \rightharpoonup u$ in $\Hr$, $u_n \to u$ in $L^p(\R^N)$ for any
  $p \in (2,2^*_s)$ and a.e. in $\R^N$. Besides, as already seen in
  the previous section, using \cite[Lemma 3]{MR695536} we get
\begin{equation} \label{eq41}
(-\Delta)^su_n+ \mu_n u_n - f(u_n) \to 0 \quad \hbox{in} \, (\Hr)^*
\end{equation}  
where
\begin{gather*}
  \mu_n:=\frac{1}{m} \left( \int_{\R^N}f(u_n)u_n \, dx-\left[
      u_n\right]_{\Hs}^2 \right).
\end{gather*}
Again, similarly to the proof of Lemma \ref{lemma17}, we can assume
the existence of $\mu \in \R$ such that $\mu_n \to \mu$, from which we
derive
\begin{equation} \label{eq42}
(-\Delta)^s+\mu u=f(u).
\end{equation}
\textit{Claim:}
$u \neq 0$.

If $u=0$, then by the compact embedding $u_n \to 0$ in
$L^{2+\frac{4s}{N}}(\R^N)$. Hence, using Lemma \ref{lemma1} $(ii)$ and
the fact that $P(u_n) \to 0$, we have $\int_{\R^N}F(u_n) \, dx \to 0$
and
\begin{gather*}
  \left[u_n \right]^2_{\Hs}=P(u_n)+ \frac{N}{2s}\int_{\R^N}
  \tilde{F}(u_n) \, dx \to 0,
\end{gather*}
from which
\begin{gather*}
  c=\lim_{n \to +\infty} I(u_n)=\frac{1}{2} \lim_{n \to +\infty}
  \left[u_n \right]_{\Hs}^2-\lim_{n \to +\infty} F(u_n) \, dx = 0,
\end{gather*}
that contradicts the hypothesis of $c>0$. Now, since $u \neq 0$, as we
obtained ($\ref{eq40}$), we get
\begin{gather*}
  \mu:=\frac{1}{ms} \int_{\R^N} \left(N F(u)-
    \frac{N-2s}{2}f(u)u \right)\, dx >0.
\end{gather*}
Since $u_n \rightharpoonup u$ in $\Hr$, by Lemma \ref{lemma23}
\begin{gather*}
\int_{\R^N} \left[ f(u_n)-f(u) \right]u \, dx \to 0.
\end{gather*}
Indeed, the fractional Sobolev inequality implies that
$u \in L^{\frac{2N}{N-2s}}(\R^N)$, and the multiplication by \(u\)
turns out to be a continuous linear operator from
\(L^{\frac{2N}{N+2s}}(\R^N)\) into \(L^1(\mathbb{R}^N)\).  Now,
observing that $\int_{\R^N}f(u_n) (u_n-u) \, dx \to 0$ by Lemma \ref{lemma1}
$(iii)$ we get
\begin{gather*}
\lim_{n \to +\infty} \int_{\R^N} f(u_n)u_n \, dx= \int_{\R^N} f(u)u \, dx.
\end{gather*}
Finally, from \eqref{eq41} and \eqref{eq42} one has
\begin{multline*}
\left[u\right]_{\Hs}^2+ \mu \int_{\R^N} u^2 \, dx  = \int_{\R^N}f(u)u \, dx \\
 = \lim_{n \to +\infty} \int_{\R^N} f(u_n) u_n \, dx = \lim_{n \to +\infty} \left[ u_n \right]_{\Hs}^2+\mu m,
\end{multline*}
and since $\mu >0$,
\begin{gather*}
\lim_{n \to +\infty} \left[ u_n \right]_{\Hs}^2=\left[u\right]_{\Hs}^2, \quad \lim_{n \to +\infty} \int_{\R^N}u_n^2 \, dx=m=\int_{\R^N} u^2 \, dx.
\end{gather*}
Thus $u_n \to u$ in $\Hr$.
\end{proof}
\begin{lemma} \label{lemma21} For any $c>0$, there exists
  $\beta=\beta(c)>0$ and $k(c) \in \N$ such that for any $k \geq k(c)$
  and any $u \in \p_m \cap \Hr$
\begin{gather*}
  \|\pi_k u\|_{H^s(\mathbb{R}^N)} \leq \beta \quad\hbox{implies}\quad
  I(u)\geq c .
\end{gather*} 
\end{lemma}
\begin{proof}
  By contradiction, we assume that there exists $c_0$ such that for
  any $\beta>0$ and any $k \in \N$ it is possible to find
  $\ell \geq k$ and $u \in \mathcal{P}_m \cap \Hr$ such that
\begin{gather*}
I(u)<c_0 \quad \hbox{with} \, \, \, \Vert \pi_k u \Vert_{\Hs} \leq \beta.
\end{gather*}
In view of that, one can find a sequence $(k_j)_j \subset \N$, with
$k_j \to \infty$ as $j \to \infty$, and a sequence
$(u_j)_j \subset \p_m \cap \Hr$ such that
\begin{equation} \label{eq43} \Vert \pi_{k_j}u_j \Vert_{\Hs}\leq
  \frac{1}{j} \quad \hbox{and} \quad I(u_j)<c_0
\end{equation}
for any $j \in \N$. Noticing that by Lemma \ref{lemma5} $(iv)$
$(u_j)_j$ is bounded, up to a subsequence we have
$u_j \rightharpoonup u$ in $\Hr$ and $L^2(\R^N)$.

\textit{Claim:} $u=0$.

Since $k_j \to \infty$, it follows that $\pi_{k_j}u \to u$ in $L^2(\R^N)$, hence
\begin{gather*}
(\pi_{k_j}u_j,u)_{L^2(\R^N)}=(u_j,\pi_{k_j}u)_{L^2(\R^N)}\to (u,u)_{L^2(\R^N)}
\end{gather*}
 as $ j \to \infty$.
\end{proof}
On the other hand, using \eqref{eq43} we get $\pi_{k_j}u_j \to 0$ in
$L^2(\R^N)$, thus the claim must hold. Now, since
$\Vert u_j \Vert_{L^{2+\frac{4s}{N}}(\R^N)} \to 0$ by the compact
embedding, $(u_j)_j \subset \p_m \cap \Hr$, and Lemma \ref{lemma1}
$(ii)$, we obtain
\begin{gather*}
\left[ u_j \right]_{\Hs}^2=\frac{N}{2s} \int_{\R^N}\tilde{F}(u_j) \, dx \to 0
\end{gather*}
as $j \to \infty$, which contradicts Lemma \ref{lemma5} $(ii)$.
\begin{lemma} \label{lemma22}
$E_{m,k} \to \infty$ as $k \to +\infty$.
\end{lemma}
\begin{proof}
We assume by contradiction that there exists $c>0$ such that 
\begin{gather*}
\liminf_{k \to +\infty} E_{m,k}<c.
\end{gather*}
Denote with $\beta(c)$ and $k(c)$ the numbers given in Lemma
\ref{lemma21}. Up to choose a bigger $c$, we can find $k>k(c)$ such
that $E_{m,k}<c$. Moreover, by definition of $E_{m,k}$ there must be
$A \in \mathcal{G}_k$ such that
\begin{gather*}
\max_{u \in A}I(\rho(u)*u)=\max_{u \in A}J(u)<c.
\end{gather*}
Now, recalling Lemma \ref{lemma4} $(iii)$ and $(iv)$ we get that the map $\varphi:A \to \p_m\cap \Hr$ defined by $\varphi(u)=\rho(u) * u$ is odd and continuous. Thus, setting $\overline{A}:=\varphi(A)\subset \p_m \cap \Hr$ we have
\begin{gather*}
\max_{v \in \overline{A}}I(v)<c
\end{gather*}
and
\begin{equation} \label{eq44}
\gamma(\overline{A})\geq \gamma(A) \geq k > k(c)
\end{equation}
by the properties of the genus. On the other hand, Lemma \ref{lemma21} implies that
\begin{gather*}
\inf_{v \in \overline{A}} \Vert \pi_{k(c)}v \Vert_{\Hs}\geq \beta(c)>0,
\end{gather*}
and after setting 
\begin{gather*}
  \phi(v):=\frac{\pi_{k(c)}v}{\Vert \pi_{k(c)} v \Vert_{\Hs}} \quad
  \hbox{for any} \, v \in \overline{A}
\end{gather*}
  we get
\begin{gather*}
\gamma(\overline{A})\leq \gamma(\phi(\overline{A})) \leq k(c)
\end{gather*}
noticing that $\phi$ is odd, continuous and that
$\phi(\overline{A}) \subset V_{k(c)}$.  That is against
\eqref{eq44}. Therefore $E_{m,k} \to \infty$ as $k \to +\infty$.
\end{proof}
\begin{proof}[Proof of Theorem \ref{th3}]
  For each $k \in \N$, by Lemmas \ref{lemma18} and \ref{lemma19} one
  can find a Palais-Smale sequence $(u_n)_n \subset \p_m \cap \Hr$ of
  the constrained functional $I_{|S_m \cap \Hr}$ at level
  $E_{m,k}>0$. By Lemma \ref{lemma5} $(u_n)_n$ is bounded and by
  virtue of Lemma \ref{lemma20} we deduce that $(P_m)$ has a radial
  solution $u_k$ such that $I(u_k)=E_{m,k}$. Moreover, using Lemma
  \ref{lemma19} $(ii)$ and Lemma \ref{lemma22}, we get
\begin{gather*}
  I(u_{k+1})\geq I(u_k)>0 \quad \hbox{for any $k \geq 1$}
\end{gather*}
and $I(u_k) \to \infty$.
\end{proof}

\section{Appendix}

\begin{proof}[Proof of Lemma \ref{lemma1}]
  $(i)$ It suffices to show that there exists $\delta>0$ such that
  \begin{gather*}
    \int_{\R^N} |F(u)| \, dx \leq \frac{1}{4} \left[ u \right]_{\Hs}^2
  \end{gather*}
  whenever $u \in B_m$ and $\left[ u \right]_{\Hs} \leq \delta$. In
  order to show that, we start noticing that $(f_0)$, $(f_1)$, and
  $(f_2)$ imply that for every $\varepsilon >0$ we can find
  $C_1=C_1(\varepsilon)>0$ such that
  \begin{equation} \label{eq1} |F(u)|\leq \varepsilon
    |t|^{2+\frac{4s}{N}}+C_1|t|^{\frac{2N}{N-2s}}.
  \end{equation}
  Hence, by \eqref{eq1}, using the interpolation inequality and the
  fractional Sobolev inequality (see for instance \cite[Theorem
  6.5]{MR2944369}), we get
\begin{align*}
  \int_{\R^N} |F(u)| \, dx & \leq \varepsilon \int_{\R^N} |u|^{2+\frac{4s}{N}} \, dx + C_1 \int_{\R^N} |u|^{\frac{2N}{N-2s}} \, dx \leq \varepsilon m^{\frac{2s}{N}} \Vert u \Vert_{L^{2^*_s}(\R^N)}^2+C_1\Vert u \Vert_{L^{2^*_s}(\R^N)}^{2^*_s}  \\
                           &\leq  \varepsilon m^{\frac{2s}{N}}C_1\left[ u \right]_{\Hs}^2 + C_2 \left[ u \right]_{\Hs}^{2^*_s} 
                             = \left[ \varepsilon m^{\frac{2s}{N}}C_1 + C_2 \left[ u \right]_{\Hs}^{2^*_s-2} \right] \left[ u\right]_{\Hs}^2.
\end{align*}
Choosing
\begin{gather*}
  \varepsilon = \frac{1}{8m^{\frac{2s}{N}}C_1} \quad \hbox{and} \quad
  \delta=\left( \frac{1}{C_2} \right)^{\frac{1}{2^*_s-2}}
\end{gather*}
the assertion is verified.

$(ii)$ Since $(f_0)$, $(f_1)$ and $(f_2)$ hold, for every
$\varepsilon >0$ there exists $C_3, \, C_4>0$ such that
\begin{gather*}
  |f(t)t| \leq
  \frac{\varepsilon}{2}|t|^{\frac{2N}{N-2s}}+C_3|t|^{2+\frac{4s}{N}}
\end{gather*}
and
\begin{gather*}
  |F(t)| \leq
  \frac{\varepsilon}{2}|t|^{\frac{2N}{N-2s}}+C_4|t|^{2+\frac{4s}{N}},
\end{gather*}
which implies
\begin{equation} \label{eq2}
  |\tilde{F}(t)| \leq
  \varepsilon|t|^{\frac{2N}{N-2s}}+\left(C_3+C_4\right)|t|^{2+\frac{4s}{N}}.
\end{equation}
By \eqref{eq2} we have
\begin{align*}
  \int_{\R^N} |\tilde{F}(u_n)| \, dx & \leq \varepsilon \int_{\R^N} |u_n|^{\frac{2N}{N-2s}} \, dx + \int_{\R^N} |u_n|^{2+ \frac{4s}{N}} \, dx \\
                                     & \leq \varepsilon C_5 \left[ u_n \right]_{\Hs}^{\frac{2N}{N-2s}} + \left(C_3+C_4\right] \Vert u_n \Vert_{L^{2+\frac{4s}{N}}(\R^N)}^{2+\frac{4s}{N}} \\
                                     & \leq \varepsilon C_6  + \left(C_3+C_4\right) \Vert u_n \Vert_{L^{2+\frac{4s}{N}}(\R^N)}^{2+\frac{4s}{N}}\to 0
\end{align*}
as $n\to +\infty$ and $\varepsilon \to 0$. The proof of \(\lim_{n \to +\infty} \int_{\mathbb{R}^N} |F(u_n)|\, dx =0\) is similarl.

$(iii)$. $(f_0)$, $(f_1)$ and $(f_2)$ imply that for every $\varepsilon >0$ we can find $C_7>0$ such that 
\begin{equation}\label{eq3}
|f(t)| \leq \varepsilon |t|^{\frac{N+2s}{N-2s}}+C_7|t|^{1+\frac{4s}{N}}.
\end{equation}
Hence, by \eqref{eq3}, we obtain that
\begin{align*}
  \int_{\R^N} |f(u_n)||v_n| \, dx & \leq \varepsilon \int_{\R^N}|u_n|^{\frac{N+2s}{N-2s}}|v_n| \, dx + C_7 \int_{\R^N} |u_n|^{1+\frac{4s}{N}}|v_n|\, dx \\
                                  &\leq \varepsilon \Vert u_n \Vert_{L^{2^*_s}(\R^N)}^{\frac{N+2s}{2N}}\Vert v_n \Vert_{L^{2^*_s}(\R^N)}^{\frac{N-2s}{2N}}+C_7 \Vert u_n \Vert_{L^{2+\frac{4s}{N}}(\R^N)}^{\frac{N+4s}{2(N+2s)}}\Vert v_n \Vert_{L^{2+\frac{4s}{N}}(\R^N)}^{\frac{N}{2(N+2s)}} \\
                                  & \leq \varepsilon C_8 \Vert u_n \Vert_{\Hs}^{\frac{N+2s}{2N}} \Vert v_n \Vert_{\Hs}^{\frac{N-2s}{2N}}+C_9 \Vert u_n \Vert_{\Hs}^{\frac{N+4s}{2(N+2s)}}\Vert v_n \Vert_{L^{2+\frac{4s}{N}}(\R^N)}^{\frac{N}{2(N+2s)}}\\
                                  & \leq \varepsilon C_{10} + C_{11}\Vert v_n \Vert_{L^{2+\frac{4s}{N}}(\R^N)}^{\frac{N}{2(N+2s)}} \to 0
\end{align*}
as $n\to +\infty$ and $\varepsilon \to 0$. This completes the proof of
the Lemma.
\end{proof}

\begin{proof}[Proof of Lemma \ref{lemma2}]
  $(i)$ Let us fix $m:=\Vert u \Vert_{\ldue}^2$. We observe that
  $\rho * u \in S_m$ and after a change of variables we obtain
\begin{gather*}
  \left[ \rho * u
  \right]_{\Hs}^2=\int_{\R^{2N}}\frac{e^{N\rho}(u(x)-u(y))^2}{|x-y|^{N+2s}}
  \, dx\, dy = e^{2\rho s} \left[ u\right]_{\Hs}^2.
\end{gather*}
By virtue of the previous computation, choosing $\rho \ll -1$, Lemma
\ref{lemma1} $(i)$ guarantees the existence of a $\delta >0$ such that
if $\left[\rho * u \right]_{\Hs} \leq \delta$ then
\begin{gather*}
  \frac{1}{4}e^{2 \rho s} \left[ u \right]_{\Hs}^2\leq I(\rho * u)
  \leq e^{2 \rho s} \left[ u \right]_{\Hs}^2,
\end{gather*}
thus 
\begin{gather*}
\lim_{\rho \to - \infty} I(\rho * u)=0^+.
\end{gather*}

$(ii)$ For every $\lambda \geq 0$ we define the function
$h_{\lambda}\colon \R \to \R$ as follows
\begin{equation}\label{eq4}
h_{\lambda}(t)=
\begin{cases}
\frac{F(t)}{|t|^{2+\frac{4s}{N}}}+\lambda & t \neq 0 \\
\lambda & t = 0.
\end{cases}
\end{equation}
It is straightforward to verify that
$F(t)=h_{\lambda}(t)|t|^{2+\frac{4s}{N}}-\lambda|t|^{2+\frac{4s}{N}}$. Moreover,
from $(f_0)$ and $(f_1)$ it follows that $h_\lambda$ is continuous,
whereas thanks to $(f_3)$ we have
\begin{gather*}
h_{\lambda}(t) \to +\infty \quad \hbox{as} \quad t\to +\infty.
\end{gather*}
Putting together the divergence of the limit above at infinity and
$(f_1)$, we can find $\lambda>0$ large enough such that
$h_{\lambda}(t) \geq 0$ for every $t \in \R$. Now, applying the well
known Fatou's Lemma, we obtain
\begin{gather*}
  \liminf_{\rho \to \infty} \int_{\R^N} h_{\lambda} (e^{\frac{N
      \rho}{2}}u)|u|^{2+\frac{4s}{N}}\, dx \geq \int_{\R^N}\lim_{\rho
    \to \infty} h_{\lambda}(e^{\frac{N
      \rho}{2}}u)|u|^{2+\frac{4s}{N}}\, dx= \infty.
\end{gather*}
Then, we observe that
\begin{align} \label{eq10} I(\rho *u) & = \frac{1}{2} \left[ \rho *u
  \right]_{\Hs}^2 + \lambda \int_{\R^N} |\rho *u|^{2+ \frac{4s}{N}} \,
  dx- \int_{\R^N}h_{\lambda}(\rho *u)|\rho *u|^{2+ \frac{4s}{N}} \, dx
  \\ \notag & =  e^{2\rho s}\left[\frac{1}{2} \left[ u \right]_{\Hs}^2
    + \lambda \int_{\R^N} |u|^{2+ \frac{4s}{N}} \, dx-
    \int_{\R^N}h_{\lambda}(e^{\frac{N \rho}{2}}u)|u|^{2+ \frac{4s}{N}} \, dx\right],
\end{align}
from which it follows immediately that
\begin{gather*}
  \lim_{\rho \to \infty} I(\rho * u )=- \infty.
\end{gather*}
\end{proof}

\bibliographystyle{plain}
\bibliography{prescribedmass}
\end{document}